\documentclass[a4paper, intlimits, reqno]{amsart}

\usepackage[english]{babel}
\usepackage[T1]{fontenc}
\usepackage[utf8]{inputenc}

\usepackage{amsmath}
\usepackage{amssymb}
\usepackage{MnSymbol}
\usepackage{amsthm}
\allowdisplaybreaks
\usepackage{amsfonts}
\usepackage{mathrsfs} 
\usepackage{mathtools}
\usepackage{enumitem}
\usepackage{multicol}
\usepackage{url}
\usepackage{dsfont}
\usepackage[numbers]{natbib}
\usepackage{caption}
\usepackage{tikz}
\usetikzlibrary{patterns,calc,arrows}
\tikzstyle{arrow}=[draw, -latex]
\usepackage{prettyref}
\usepackage{doi}

\newrefformat{def}{Definition \ref{#1}}
\newrefformat{rem}{Remark \ref{#1}}
\newrefformat{sect}{Section \ref{#1}}
\newrefformat{prop}{Proposition \ref{#1}}
\newrefformat{thm}{Theorem \ref{#1}}
\newrefformat{cor}{Corollary \ref{#1}}
\newrefformat{ex}{Example \ref{#1}}
\newrefformat{cond}{Condition \ref{#1}}
\newrefformat{conv}{Convention \ref{#1}}

\swapnumbers
\newtheoremstyle{dotless}{}{}{\itshape}{}{\bfseries}{}{}{}
\theoremstyle{dotless}
\theoremstyle{plain}
\newtheorem{thm}{Theorem}[section]
\newtheorem{lem}[thm]{Lemma}
\newtheorem{prop}[thm]{Proposition}
\newtheorem{cor}[thm]{Corollary}
\theoremstyle{definition}
\newtheorem{defn}[thm]{Definition}
\newtheorem{rem}[thm]{Remark}

\newtheorem*{cond}{Condition}{\bf}{\rm}
\newtheorem*{condCT}{Condition (CT)}{\bf}{\rm}
\newtheorem*{condH3CT}{Condition (H3CT)}{\bf}{\rm}

\newcommand{\N} {\mathbb{N}}
\newcommand{\R} {\mathbb{R}}
\newcommand{\C} {\mathbb{C}}

\DeclareMathOperator{\id}{id}
\DeclareMathOperator{\re}{Re}
\DeclareMathOperator{\im}{Im}

\providecommand{\differential}{\mathrm{d}}
\renewcommand{\d}{\differential}

\newcommand{\vertiii}[1]{{\left\vert\kern-0.25ex\left\vert\kern-0.25ex\left\vert #1 
    \right\vert\kern-0.25ex\right\vert\kern-0.25ex\right\vert}}  

\begin{document}

\title[The inhomogeneous Cauchy-Riemann equation]{The inhomogeneous Cauchy-Riemann equation for weighted smooth 
vector-valued functions on strips with holes}
\author[K.~Kruse]{Karsten Kruse}
\address{Hamburg University of Technology \\ Institute of Mathematics \\
Am Schwarzenberg-Campus~3 \\
21073 Hamburg \\
Germany}
\email{karsten.kruse@tuhh.de}



\date{\today}
\begin{abstract}
This paper is dedicated to the question of surjectivity of the Cauchy-Riemann operator on spaces 
$\mathcal{EV}(\Omega,E)$ of $\mathcal{C}^{\infty}$-smooth vector-valued functions whose 
growth on strips along the real axis with holes $K$ is induced
by a family of continuous weights $\mathcal{V}$. Vector-valued means that these functions have values in 
a locally convex Hausdorff space $E$ over $\C$.  
We characterise the weights $\mathcal{V}$ which give a counterpart of the Grothendieck-K\"othe-Silva duality 
$\mathcal{O}(\C\setminus K)/\mathcal{O}(\C)\cong\mathscr{A}(K)$ with non-empty compact $K\subset\R$ 
for weighted holomorphic functions. 
We use this duality and splitting theory to prove the surjectivity of the Cauchy-Riemann operator
$\overline{\partial}\colon\mathcal{EV}(\Omega,E)\to\mathcal{EV}(\Omega,E)$ 
for certain $E$. 
This solves the smooth (holomorphic, distributional) parameter dependence problem for the Cauchy-Riemann operator 
on $\mathcal{EV}(\Omega,\C)$.
\end{abstract}
\maketitle
\section{Introduction}
The smooth (holomorphic, distributional) parameter dependence problem for the Cauchy-Riemann operator 
$\overline{\partial}:=(1/2)(\partial_{1}+i\partial_{2})$ on the space $\mathcal{C}^{\infty}(\Omega)$ of 
smooth complex-valued functions on an open set $\Omega\subset\R^{2}$
is whether for every family $(f_{\lambda})_{\lambda\in U}$ in $\mathcal{C}^{\infty}(\Omega)$ 
depending smoothly (holomorphically, distributionally) on a parameter $\lambda$ in an open set $U\subset\R^{d}$
there is a family $(u_{\lambda})_{\lambda\in U}$ in $\mathcal{C}^{\infty}(\Omega)$ with the same kind of parameter dependence 
such that
\[
\overline{\partial}u_{\lambda}=f_{\lambda},\quad \lambda\in U.
\]
Here, smooth (holomorphic, distributional) parameter dependence of $(f_{\lambda})_{\lambda\in U}$ means that 
the map $\lambda\mapsto f_{\lambda}(x)$ is an element of $\mathcal{C}^{\infty}(U)$ (of the space of holomorphic functions 
$\mathcal{O}(U)$ on $U\subset\C$ open, the space of distributions $\mathcal{D}(V)'$ for open $V\subset\R^{d}$ where 
$U=\mathcal{D}(V)$) for each $x\in\Omega$.

The parameter dependence problem for a variety of partial differential operators on several spaces of 
(generalised) differentiable functions has been extensively studied, 
see e.g.\ \cite{Dom1, dierolf2014, dierolf_sieg2017, vogt1983, V1, kalmes2018} 
and the references and background in \cite{bonetdomanski2006, kruse2019_1}.
The answer to this problem for the Cauchy-Riemann operator is affirmative since the Cauchy-Riemann operator 
\begin{equation}\label{CR_classic}
 \overline{\partial}^{E}\colon \mathcal{C}^{\infty}(\Omega,E)\to\mathcal{C}^{\infty}(\Omega,E)
\end{equation}
on the space $\mathcal{C}^{\infty}(\Omega,E)$ of $E$-valued smooth functions is surjective if $E=\mathcal{C}^{\infty}(U)$ 
($\mathcal{O}(U)$, $\mathcal{D}(V)'$) by \cite[Corollary 3.9, p.\ 1112]{D/L} 
which is a consequence of the splitting theory of Bonet and Doma\'nski for PLS-spaces \cite{bonetdomanski2006, Dom1}, 
the topological isomorphy of $\mathcal{C}^{\infty}(\Omega,E)$ to Schwartz' $\varepsilon$-product 
$\mathcal{C}^{\infty}(\Omega)\varepsilon E$ and the fact that 
$\overline{\partial}\colon\mathcal{C}^{\infty}(\Omega)\to\mathcal{C}^{\infty}(\Omega)$ is surjective on the 
nuclear Fr\'echet space $\mathcal{C}^{\infty}(\Omega)$ (with its usual topology).
More generally, the Cauchy-Riemann operator \eqref{CR_classic} is surjective if $E$ is a Fr\'echet space 
by Grothendieck's classical theory of tensor products \cite{Gro} or if $E:=F_{b}'$ where $F$ is a Fr\'{e}chet space 
satisfying the condition $(DN)$ by \cite[Theorem 2.6, p.\ 174]{vogt1983} or if $E$ is an ultrabornological 
PLS-space having the property $(PA)$ by \cite[Corollary 3.9, p.\ 1112]{D/L} since $\operatorname{ker}\overline{\partial}$ 
has the property $(\Omega)$ by \cite[Proposition 2.5 (b), p.\ 173]{vogt1983}. 
The first and the last result cover the case that $E=\mathcal{C}^{\infty}(U)$ or $\mathcal{O}(U)$ whereas 
the last covers the case $E=\mathcal{D}(V)'$ as well. 
More examples of the second or third kind of such spaces $E$ are arbitrary Fr\'echet-Schwartz spaces, 
the space $\mathcal{S}(\R^{d})'$ of tempered distributions, 
the space $\mathcal{D}(V)'$ of distributions, the space $\mathcal{D}_{(w)}(V)'$ of ultradistributions of Beurling type 
and some more (see \cite{Dom1}, \cite[Corollary 4.8, p.\ 1116]{D/L} and \cite[Example 3, p.\ 7]{kruse2019_1}).
 
In this paper we consider the Cauchy-Riemann operator on spaces $\mathcal{EV}(\Omega,E)$ of weighted smooth $E$-valued functions where 
$E$ is a locally convex Hausdorff space over $\C$ with a system of seminorms $(p_{\alpha})_{\alpha\in\mathfrak{A}}$ 
generating its topology. These spaces consist of functions $f\in\mathcal{C}^{\infty}(\Omega,E)$ 
fulfilling additional growth conditions induced by a family $\mathcal{V}:=(\nu_{n})_{n\in\N}$ of continuous functions 
$\nu_{n}\colon \Omega\to (0,\infty)$ on a sequence of open sets $(\Omega_{n})_{n\in\N}$ with $\Omega=\bigcup_{n\in\N}\Omega_{n}$
given by the constraint 
\[
|f|_{n,m,\alpha}:=\sup_{\substack{x\in \Omega_{n}\\ \beta\in\N^{2}_{0},\,|\beta|\leq m}}
p_{\alpha}\bigl((\partial^{\beta})^{E}f(x)\bigr)\nu_{n}(x)<\infty
\]
for every $n\in\N$, $m\in\N_{0}$ and $\alpha\in\mathfrak{A}$. The aim is to derive sufficient conditions on 
$\mathcal{V}$ and $(\Omega_{n})_{n\in\N}$ such that 
\[
 \overline{\partial}^{E}\colon \mathcal{EV}(\Omega,E)\to\mathcal{EV}(\Omega,E)
\]
is surjective if $E:=F_{b}'$ where $F$ is a Fr\'{e}chet space satisfying the condition $(DN)$ or if $E$ is an ultrabornological 
PLS-space having the property $(PA)$.  

In \cite{kruse2018_5, kruse2019_5} this was done in the case that $E$ is a Fr\'echet space using conditions on 
$\mathcal{V}$ and $(\Omega_{n})_{n\in\N}$ which guarantee that $\mathcal{EV}(\Omega)$ is a nuclear Fr\'echet space, 
$\mathcal{EV}(\Omega,E)$ is topological isomorphic to $\mathcal{EV}(\Omega)\varepsilon E$ for complete $E$ 
and $\overline{\partial}\colon\mathcal{EV}(\Omega)\to\mathcal{EV}(\Omega)$ is surjective. 
By proving that $\operatorname{ker}\overline{\partial}$ has property $(\Omega)$ under some additional assumptions on $\mathcal{V}$ 
this was extended in \cite{kruse2019_1} to $E:=F_{b}'$ where $F$ is a Fr\'{e}chet space satisfying the condition $(DN)$ or 
ultrabornological PLS-spaces $E$ with $(PA)$ in the case that the $\Omega_{n}$ 
are strips along the real axis, i.e.\ $\Omega_{n}:=\{z\in\C\;|\;|\im(z)|<n\}$ for $n\in\N$ 
(see \cite[Corollary 17, p.\ 21]{kruse2019_1}). In particular, these conditions are satisfied 
if $\nu_{n}(z):=\exp(a_{n}|\re(z)|^{\gamma})$, $z\in\C$, for some $0<\gamma\leq 1$ and $a_{n}\nearrow 0$ 
by \cite[Corollary 18, p.\ 21]{kruse2019_1}.
In the present paper we consider the case that the $\Omega_{n}$ are strips along the real axis with holes around 
non-empty compact sets $K\subset[-\infty,\infty]$ and we are confronted with the task of 
deriving sufficient conditions on $\mathcal{V}$ such that $\operatorname{ker}\overline{\partial}$ 
has $(\Omega)$. The corresponding spaces $\mathcal{EV}(\Omega,E)$ and their subspaces of holomorphic functions 
are of interest because they are the basic spaces for the theory of vector-valued Fourier hyperfunctions, 
see e.g.\ \cite{Ion/Ka, Ito/Nag, J, Kawai, ich, L2, L3}.

Let us summarise the content of our paper. In Section 2 we recall necessary definitions 
and preliminaries which are needed in the subsequent sections. 
Section 3 is dedicated to a counterpart for weighted holomorphic functions 
of the Silva-K\"othe-Grothendieck duality
\[
\mathcal{O}(\C\setminus K)/\mathcal{O}(\C)\cong \mathscr{A}(K)_{b}'
\]
where $K\subset\R$ is a non-empty compact set 
and $\mathscr{A}(K)$ the space of germs of real analytic functions on $K$ 
(see \prettyref{thm:duality}, \prettyref{cor:duality}, \prettyref{cor:duality_example}). 
In Section 4 we use this duality to characterise the weights $\mathcal{V}$ such 
that the kernel $\operatorname{ker}\overline{\partial}$ satisfies property $(\Omega)$ 
in the case that $(\Omega_{n})_{n\in\N}$ is a sequence of strips along the real axis 
with holes around a non-empty compact set $K\subset[-\infty,\infty]$ 
(see \prettyref{thm:Omega_strips_with_holes}, \prettyref{cor:Omega_strips_with_holes}). 
The preceding conditions on $\mathcal{V}$ are used in Section 5 
to obtain the surjectivity of the Cauchy-Riemann operator on $\mathcal{EV}(\Omega,E)$ 
in the case that $(\Omega_{n})_{n\in\N}$ is a sequence of strips along the real axis 
with holes around $K$ for $E:=F_{b}'$ where $F$ is a Fr\'{e}chet space 
satisfying the condition $(DN)$ or an ultrabornological PLS-space $E$ having the property $(PA)$ 
(see \prettyref{thm:surj_CR_DN_PA}). Especially, these conditions hold 
if $\nu_{n}(z):=\exp(a_{n}|\re(z)|^{\gamma})$, $z\in\C$, for some $0<\gamma\leq 1$ and $a_{n}\nearrow 0$ 
(see \prettyref{cor:surj_CR_DN_PA_exa}).
\section{Notation and Preliminaries}
The notation and preliminaries are essentially the same as in \cite[Section 2]{kruse2017, kruse2018_5, kruse2019_1}.
We denote by $\overline{\R}:=\R\cup\{\pm\infty\}$ the two-point compactifaction of $\R$ and set $\overline{\C}:=\overline{\R}+i\R$.
We define the distance of two subsets $M_{0}, M_{1} \subset\R^{2}$ w.r.t.\ the Euclidean norm $|\cdot|$ on $\R^{2}$ via
\[
  \d(M_{0},M_{1}) 
:=\begin{cases}
   \inf_{x\in M_{0},\,y\in M_{1}}|x-y| &,\;  M_{0},\,M_{1} \neq \emptyset, \\
   \infty &,\;  M_{0}= \emptyset \;\text{or}\; M_{1}=\emptyset.
  \end{cases}
\]
Moreover, we denote by $\mathbb{B}_{r}(x):=\{w\in\R^{2}\;|\;|w-x|<r\}$ the Euclidean ball around $x\in\R^{2}$ 
with radius $r>0$ and identify $\R^{2}$ and $\C$ as (normed) vector spaces.
We denote the complement of a subset $M\subset \R^{2}$ by $M^{C}:= \R^{2}\setminus M$, 
the closure of $M$ by $\overline{M}$ and the boundary of $M$ by $\partial M$.
For a function $f\colon M\to\C$ and $K\subset M$ we denote by $f_{\mid K}$ the restriction of $f$ to $K$ and by 
\[
 \|f\|_{K}:=\sup_{x\in K}|f(x)|
\]
the sup-norm on $K$. By $\mathcal{C}(\Omega)$ we denote the space of continuous $\C$-valued functions on a set $\Omega\subset\R^{2}$ and 
by $L^{1}(\Omega)$ the space of (equivalence classes of) $\C$-valued Lebesgue integrable functions on 
a measurable set $\Omega\subset\R^{2}$. 

By $E$ we always denote a non-trivial locally convex Hausdorff space over the field 
$\C$ equipped with a directed fundamental system of seminorms $(p_{\alpha})_{\alpha\in \mathfrak{A}}$. 
If $E=\C$, then we set $(p_{\alpha})_{\alpha\in \mathfrak{A}}:=\{|\cdot|\}$.
We recall that for a disk $D\subset E$, i.e.\ a bounded, absolutely convex set, 
the vector space $E_{D}:=\bigcup_{n\in\N}nD$ becomes a normed space if it is equipped with 
gauge functional of $D$ as a norm (see \cite[p.\ 151]{Jarchow}). The space $E$ is called locally 
complete if $E_{D}$ is a Banach space for every closed disk $D\subset E$ (see \cite[10.2.1 Proposition, p.\ 197]{Jarchow}).
Further, we denote by $L(F,E)$ the space of continuous linear maps from 
a locally convex Hausdorff space $F$ to $E$ and sometimes use the notation $\langle T, f\rangle:=T(f)$, $f\in F$,  
for $T\in L(F,E)$. If $E=\C$, we write $F':=L(F,\C)$ for the dual space of $F$. 
If $F$ and $E$ are (linearly topologically) isomorphic, we write $F\cong E$.
We denote by $L_{b}(F,E)$ the space $L(F,E)$ equipped with the locally convex topology of uniform convergence 
on the bounded subsets of $F$.

We recall the following well-known definitions concerning continuous partial differentiability of 
vector-valued functions (c.f.\ \cite[p.\ 237]{kruse2018_2}). A function $f\colon\Omega\to E$ on an open set 
$\Omega\subset\R^{2}$ to $E$ is called continuously partially differentiable ($f$ is $\mathcal{C}^{1}$) 
if for the $n$th unit vector $e_{n}\in\R^{2}$ the limit
\[
(\partial^{e_{n}})^{E}f(x):=(\partial_{n})^{E}f(x)
:=\lim_{\substack{h\to 0\\ h\in\R, h\neq 0}}\frac{f(x+he_{n})-f(x)}{h}
\]
exists in $E$ for every $x\in\Omega$ and $(\partial^{e_{n}})^{E}f$ 
is continuous on $\Omega$ ($(\partial^{e_{n}})^{E}f$ is $\mathcal{C}^{0}$) for every $n\in\{1,2\}$. 
For $k\in\N$ a function $f$ is said to be $k$-times continuously partially differentiable 
($f$ is $\mathcal{C}^{k}$) if $f$ is $\mathcal{C}^{1}$ and all its first partial derivatives are $\mathcal{C}^{k-1}$.
A function $f$ is called infinitely continuously partially differentiable ($f$ is $\mathcal{C}^{\infty}$) 
if $f$ is $\mathcal{C}^{k}$ for every $k\in\N$.
The linear space of all functions $f\colon\Omega\to E$ which are $\mathcal{C}^{\infty}$ 
is denoted by $\mathcal{C}^{\infty}(\Omega,E)$. 
Let $f\in\mathcal{C}^{\infty}(\Omega,E)$. For $\beta=(\beta_{n})\in\N_{0}^{2}$ we set 
$(\partial^{\beta_{n}})^{E}f:=f$ if $\beta_{n}=0$, and
\[
(\partial^{\beta_{n}})^{E}f
:=\underbrace{(\partial^{e_{n}})^{E}\cdots(\partial^{e_{n}})^{E}}_{\beta_{n}\text{-times}}f
\]
if $\beta_{n}\neq 0$ as well as 
\[
(\partial^{\beta})^{E}f
:=(\partial^{\beta_{1}})^{E}(\partial^{\beta_{2}})^{E}f.
\]
Due to the vector-valued version of Schwarz' theorem $(\partial^{\beta})^{E}f$ is independent of the order of the partial 
derivatives on the right-hand side, we call $|\beta|:=\beta_{1}+\beta_{2}$ the order of differentiation 
and write $\partial^{\beta}f:=(\partial^{\beta})^{\C}f$. 

A function $f\colon\Omega\to E$ on an open set 
$\Omega\subset\C$ to $E$ is called holomorphic if the limit
\[
\bigl(\frac{\partial}{\partial z}\bigr)^{E}f(z_{0})
:=\lim_{\substack{h\to 0\\ h\in\C, h\neq 0}}\frac{f(z_{0}+h)-f(z_{0})}{h}
\]
exists in $E$ for every $z_{0}\in\Omega$. As before we define derivatives of higher order recursively, i.e.\ 
for $n\in\N_{0}$ we set $((\frac{\partial}{\partial z})^{0})^{E}f:=f$ and 
$((\frac{\partial}{\partial z})^{n})^{E}f:=(\frac{\partial}{\partial z})^{E}((\frac{\partial}{\partial z})^{n-1})^{E}f$, 
$n\geq 1$, if the corresponding limits exist. 
Further, we write $f^{(n)}:=((\frac{\partial}{\partial z})^{n})^{\C}f$. 
The linear space of all functions $f\colon\Omega\to E$ which are holomorphic 
is denoted by $\mathcal{O}(\Omega,E)$. 
If $E$ is locally complete and $f\in\mathcal{O}(\Omega,E)$, then $((\frac{\partial}{\partial z})^{n})^{E}f(z_{0})$ exists in $E$ 
for every $z_{0}\in\Omega$ and $n\in\N_{0}$ by \cite[2.2 Theorem and Definition, p.\ 18]{grosse-erdmann1992} 
and \cite[5.2 Theorem, p.\ 35]{grosse-erdmann1992}.
Now, the precise definition of the spaces of weighted smooth resp.\ holomorphic vector-valued functions 
from the introduction reads as follows. 

\begin{defn}[{\cite[3.2 Definition, p.\ 238]{kruse2018_2}}]\label{def:smooth_weighted_space}
Let $\Omega\subset\R^{2}$ be open and $(\Omega_{n})_{n\in\N}$ a family of non-empty
open sets such that $\Omega_{n}\subset\Omega_{n+1}$ and $\Omega=\bigcup_{n\in\N} \Omega_{n}$.
Let $\mathcal{V}:=(\nu_{n})_{n\in\N}$ be a countable family of positive continuous functions 
$\nu_{n}\colon \Omega \to (0,\infty)$ such that $\nu_{n}\leq\nu_{n+1}$ for all $n\in\N$.
We call $\mathcal{V}$ a directed family of continuous weights on $\Omega$ and set for $n\in\N$ 
\begin{enumerate}
 \item [a)]
 \[
\mathcal{E}\nu_{n}(\Omega_{n}, E):= \{ f \in \mathcal{C}^{\infty}(\Omega_{n}, E)\; | \;
\forall\;\alpha\in\mathfrak{A},\,m \in \N_{0}^{2}:\; |f|_{n,m,\alpha} < \infty \}
\]
and 
\[
\mathcal{EV}(\Omega, E):=\{ f\in \mathcal{C}^{\infty}(\Omega, E)\; | \;\forall\; n \in \N:
\; f_{\mid\Omega_{n}}\in \mathcal{E}\nu_{n}(\Omega_{n}, E)\}
\]
where
\[
|f|_{n,m,\alpha}:=\sup_{\substack{x \in \Omega_{n}\\ \beta \in \N_{0}^{2}, \, |\beta| \leq m}}
p_{\alpha}\bigl((\partial^{\beta})^{E}f(x)\bigr)\nu_{n}(x).
\]
\item [b)] 
\[
\mathcal{E}\nu_{n,\overline{\partial}}(\Omega_{n},E):= \{ f \in \mathcal{E}\nu_{n}(\Omega_{n}, E)\; | \;
f\in\operatorname{ker}\overline{\partial}^{E} \}
\]
and 
\[
\mathcal{EV}_{\overline{\partial}}(\Omega, E):=\{f\in\mathcal{EV}(\Omega, E)\;|\;
f\in\operatorname{ker}\overline{\partial}^{E}\}.
\]
\item [c)] 
\[
\mathcal{O}\nu_{n}(\Omega_{n},E):= \{ f \in \mathcal{O}(\Omega_{n}, E)\; | \;
\forall\;\alpha\in\mathfrak{A}:\;|f|_{n,\alpha} < \infty  \}
\]
and 
\[
\mathcal{OV}(\Omega, E):=\{ f\in \mathcal{O}(\Omega, E)\; | \;\forall\; n \in \N:
\; f_{\mid\Omega_{n}}\in \mathcal{O}\nu_{n}(\Omega_{n}, E)\}
\]
where 
\[
|f|_{n,\alpha}:=\sup_{\substack{x \in \Omega_{n}}}
p_{\alpha}(f(x))\nu_{n}(x).
\]
\end{enumerate}
The subscript $\alpha$ in the notation of the seminorms is omitted in the $\C$-valued case. 
The letter $E$ is omitted in the case $E=\C$ as well, e.g.\ we write
$\mathcal{E}\nu_{n}(\Omega_{n}):=\mathcal{E}\nu_{n}(\Omega_{n},\C)$ 
and $\mathcal{EV}(\Omega):=\mathcal{EV}(\Omega,\C)$ .
\end{defn}
\section{Duality}
We recall the well-known topological Silva-K\"othe-Grothendieck isomorphy 
\begin{equation}\label{eq:duality_non_weighted}
\mathcal{O}(\C\setminus K,E)/\mathcal{O}(\C,E)\cong L_{b}(\mathscr{A}(K),E)
\end{equation}
where $E$ is a quasi-complete locally convex Hausdorff space, $\varnothing\neq K\subset\R$ is compact, $\mathcal{O}(\C\setminus K,E)$ 
is equipped with the topology of uniform convergence on compact subsets of $\C\setminus K$, 
the quotient space with the induced quotient topology 
and $\mathscr{A}(K)$ is the space of germs of real analytic functions on $K$ with its inductive limit topology 
(see e.g.\ \cite[p.\ 6]{SebastiaoeSilva1950}, \cite[Proposition 1, p.\ 46]{Grothendieck1953}, 
\cite[Satz 9, p.\ 90]{Tillmann1955}, \cite[\S27.4, p.\ 375-378]{Koethe1969}, 
\cite[Theorem 2.1.3, p.\ 25]{Mori1}). 
The aim of this section is to prove a counterpart of this isomorphy 
for weighted vector-valued holomorphic functions and non-empty compact $K\subset\overline{\R}$. 

For a compact set $K\subset \overline{\R}$ and $t\in\R$, $t\geq 1$, we define the open sets
\begin{align*}
U_{t}(K)&:=\phantom{\cup} \{ z \in \C\;| \; \d(\{z\},K\cap\C)< 1/t\} \\
&\phantom{:=}\cup
\begin{cases}
\varnothing &, K \subset \R, \\
(t,\infty)+i (-1/t, 1/t) &, \infty \in K, \, -\infty \notin K, \\ 
(-\infty, -t)+i (-1/t, 1/t) &,\infty \notin K, \,  -\infty \in K,\\ 
\bigl((-\infty, -t)\cup (t,\infty)\bigr)+i (-1/t, 1/t) &, \pm\infty \in K,  
\end{cases}
\end{align*}
\begin{center}
 \includegraphics[scale=0.85]{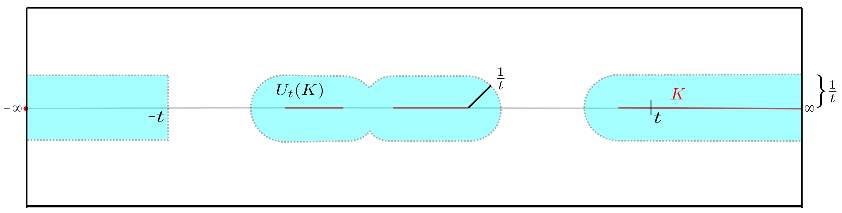}
 \captionsetup{type=figure}
 \caption{$U_{t}(K)$ for $\pm\infty\in K$ (c.f.\ \cite[Figure 3.1, p.\ 11]{ich})}
\end{center}
and
\[
S_{t}(K):= \left(\overline{U_{t}(K)}\right)^{C} \cap \{z\in \C \; | \; |\im(z)| < t \},\quad t>1, 
\quad\text{and}\quad S_{1}(K):=S_{3/2}(K)
\]
where the closure and the complement are taken in $\C$. The definition of $S_{1}(K)$ is motivated by 
\[
 \left(\overline{U_{1}(\overline{\R})}\right)^{C} \cap \{z\in \C \; | \; |\im(z)| < 1 \}=\varnothing. 
\]
\begin{center}
 \includegraphics[scale=0.85]{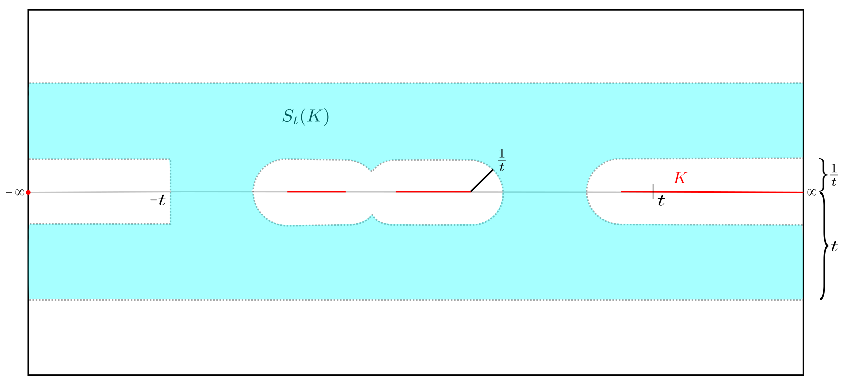}
 \captionsetup{type=figure}
 \caption{$S_{t}(K)$ for $\pm\infty\in K$ (c.f.\ \cite[Figure 3.2, p.\ 12]{ich})}
\end{center}

\begin{defn}
Let $K\subset \overline{\R}$ be a compact set, $\mathcal{V}:=(\nu_{n})_{n\in\N}$ 
a directed family of continuous weights on $\C$ and $E$ a locally convex Hausdorff space. 
Using \prettyref{def:smooth_weighted_space}, we set
\[
\mathcal{EV}(\overline{\C}\setminus K, E):=\mathcal{EV}(\C\setminus K, E)\quad\text{and}\quad 
\mathcal{OV}(\overline{\C}\setminus K, E):=\mathcal{OV}(\C\setminus K, E)
\]
with $\Omega_{n}:=S_{n}(K)$ for all $n\in\N$ and 
\[
|f|_{K,n,m,\alpha}:=|f|_{n,m,\alpha}, \;\; f\in\mathcal{EV}(\overline{\C}\setminus K, E),
\;\;\text{and}\;\;
|f|_{K,n,\alpha}:=|f|_{n,\alpha}, \;\; f\in\mathcal{OV}(\overline{\C}\setminus K, E),
\]
for $n\in\N$, $m\in\N_{0}$ and $\alpha\in\mathfrak{A}$. 
We omit the index $K$ in $|f|_{K,n,\alpha,m}$ and $|f|_{K,n,\alpha}$ if no confusion seems to be likely.
\end{defn}

The spaces $\mathcal{OV}(\overline{\C}\setminus K, E)$ play the counterpart of $\mathcal{O}(\C\setminus K,E)$ 
for our version of the isomorphy \eqref{eq:duality_non_weighted}.
Next, we introduce some conditions which guarantee the existence of a counterpart of $\mathscr{A}(K)$ 
for our purpose.

\begin{cond}\label{cond:DFS}
Let $K\subset \overline{\R}$ be a non-empty compact set and $\mathcal{V}:=(\nu_{n})_{n\in\N}$ 
a directed family of continuous weights on $\C$.
For every $n\in\N$ let 
\begin{enumerate}
 \item [$(qV_{\infty})$] there be $\mathcal{I}_{1}(n)>n$ such that for every $\varepsilon>0$ there is a compact set $Q\subset \overline{U_{\mathcal{I}_{1}(n)}(K)}$ with 
 $\nu_{n}(z)\leq\varepsilon\nu_{\mathcal{I}_{1}(n)}(z)$ for all $z\in\overline{U_{\mathcal{I}_{1}(n)}(K)}\setminus Q$.
 \item [$(qL^{1})$] $\nu_{n}(z)=\nu_{n}(|\re(z)|)$ for all $z\in\C$. In addition, if $K\cap\{\pm\infty\}\neq \varnothing$, let there be $\mathcal{I}_{2}(n)>n$ such that 
 $\frac{\nu_{n}}{\nu_{\mathcal{I}_{2}(n)}}\in L^{1}([0,\infty))$ .
\end{enumerate}
\end{cond}

Condition $(qV_{\infty})$ means that the \emph{quotient} $\tfrac{\nu_{n}}{\nu_{\mathcal{I}_{1}(n)}}$ 
\emph{vanishes at infinity} whereas $(qL^{1})$ means that the \emph{quotient} 
$\tfrac{\nu_{n}}{\nu_{\mathcal{I}_{2}(n)}}$ is an $L^{1}$\emph{-function} if $K\cap\{\pm\infty\}\neq \varnothing$.
Let $\Omega\subset\C$ be open and $f\in\mathcal{O}(\Omega)$. For $z\in\Omega$ and $n\in\N_{0}$ we denote the point evaluation 
of the $n$th complex derivative at $z$ by $\delta^{(n)}_{z}f:=f^{(n)}(z)$.

\begin{prop}\label{prop:DFS}
Let $K\subset \overline{\R}$ be a non-empty compact set and $\mathcal{V}:=(\nu_{n})_{n\in\N}$ a directed family 
of continuous weights on $\C$. 
For $n\in\N$ let
\[
\mathcal{O}\nu_{n}^{-1}(\overline{U_{n}(K)}):=\{f\in\mathcal{O}(U_{n}(K))\cap\mathcal{C}(\overline{U_{n}(K)})\;|\; 
\|f\|_{K,n}:=\|f\|_{n} < \infty \}
\]
where
\[
\|f\|_{K,n}:=\|f\|_{n}:=\sup_{z \in \overline{U_{n}(K)}}|f(z)|\nu_{n}(z)^{-1}
\]
and the spectral maps for $n,k\in\N$, $n\leq k$, be given by the restrictions
\[
\pi_{n,k}\colon \mathcal{O}\nu_{n}^{-1}(\overline{U_{n}(K)})\to\mathcal{O}\nu_{k}^{-1}(\overline{U_{k}(K)}), \;
\pi_{n,k}(f):=f_{\mid U_{k}(K)}.
\]
If $\mathcal{V}$ fulfils $(qV_{\infty})$, then 
\begin{enumerate}
\item [a)]
the inductive limit 
\[
 \mathcal{OV}^{-1}_{\operatorname{ind}}(K):=\lim_{\substack{\longrightarrow\\n\in \N}}\mathcal{O}\nu_{n}^{-1}(\overline{U_{n}(K)})
\]
exists and is a DFS-space.
\item [b)]
the span of the set of point evaluations of complex derivatives $\{\delta_{x_{0}}^{(n)} \; | \; x_{0} \in K\cap\R,\, n\in \N_{0}\}$
is dense in $\mathcal{OV}^{-1}_{\operatorname{ind}}(K)_{b}'$ if $K\subset\R$ or $K\cap\{\pm\infty\}$ contains 
no isolated points in $K$. 
\item [c)]
the span of the set of point evaluations $\{\delta_{x_{0}} \; | \; x_{0} \in K\cap\R\}$
is dense in $\mathcal{OV}^{-1}_{\operatorname{ind}}(K)_{b}'$ if $K$ has no isolated points.
\end{enumerate}
\end{prop}
\begin{proof}
$a)(i)$ First, we prove that the normed space $\mathcal{O}\nu_{n}^{-1}(\overline{U_{n}(K)})$ is a Banach space.
Let $(f_{k})_{k\in\N}$ be a Cauchy sequence in $\mathcal{O}\nu_{n}^{-1}(\overline{U_{n}(K)})$. 
Let $\varepsilon >0$ and $M\subset \overline{U_{n}(K)}$ be compact. 
Then there exists $N\in\N$ such that for all $k,m\geq N$
\begin{align*}
\varepsilon &> \|f_{k}-f_{m} \|_{n}=\sup_{z \in \overline{U_{n}(K)}}|f_{k}(z)-f_{m}(z)|\nu_{n}(z)^{-1}\\
& \geq (\|\nu_{n}\|_{M})^{-1} \|f_{k}-f_{m}\|_{M}.
\end{align*}
Thus $(f_{k})_{k\in\N}$ is also a Cauchy sequence in the Fr\'echet space 
$\mathcal{O}(U_{n}(K))\cap\mathcal{C}(\overline{U_{n}(K)})$ 
equipped with the topology induced by 
the system of seminorms $\|\cdot\|_{M}$ with compact $M\subset\overline{U_{n}(K)}$.
Therefore it converges to $f\in\mathcal{O}(U_{n}(K))\cap\mathcal{C}(\overline{U_{n}(K)})$.
Since every Cauchy sequence is bounded, there exists $C(n)\geq 0$ 
with $|f_{k}(z)|\nu_{n}(z)^{-1}\leq C(n)$ for all $z\in\overline{U_{n}(K)}$ 
and $k\in\N$, implying $f\in\mathcal{O}\nu_{n}^{-1}(\overline{U_{n}(K)})$ by pointwise convergence. 
Using the pointwise convergence again, we get for all $z\in\overline{U_{n}(K)}$ and $k\geq N$
\begin{align*}
 |f_{k}(z)-f(z)|\nu_{n}(z)^{-1} 
&=\lim_{m \to \infty}|f_{k}(z)-f_{m}(z)|\nu_{n}(z)^{-1}\\
&\leq \lim_{m \to \infty} \|f_{k}-f_{m}\|_{n}
 \leq\varepsilon
\end{align*}
and therefore $\|f_{k}-f\|_{n}\leq\varepsilon$, which proves that $\mathcal{O}\nu_{n}^{-1}(\overline{U_{n}(K)})$ is a Banach space.

$(ii)$ The maps $\pi_{n,m}\colon\mathcal{O}\nu_{n}^{-1}(\overline{U_{n}(K)})\rightarrow\mathcal{O}\nu_{m}^{-1}(\overline{U_{m}(K)})$, 
$n\leq m$, are injective by virtue of the identity theorem and the definition of sets $U_{n}(K)$. 
Thus the considered spectrum is an embedding spectrum.

%
$(iii)$ For all $M \subset U_{n}(K)$ compact and 
$f\in B_{n}:=\{g\in\mathcal{O}\nu_{n}^{-1}(\overline{U_{n}(K)})\;|\;\|g\|_{n}\leq 1\}$ we have
\[
\|f\|_{M}=\sup_{z\in M}|f(z)|\nu_{n}(z)^{-1}\nu_{n}(z) \leq \|\nu_{n}\|_{M}\|f\|_{n} \leq \|\nu_{n}\|_{M}.
\]
Thus $B_{n}$ is bounded in $\mathcal{O}(U_{n}(K))$ w.r.t.\ the system of seminorms generated by $\|\cdot\|_{M}$ for 
compact $M \subset U_{n}(K)$. 
As this space is a Fr\'echet-Montel space, 
$B_{n}$ is relatively compact and hence relatively sequentially compact in $\mathcal{O}(U_{n}(K))$.

$(iv)$ What remains to be shown is that for all $n\in\N$ there exists $m>n$ such that 
$\pi_{n,m}$ is a compact map. Because the considered spaces are Banach spaces, it suffices
to prove the existence of $m>n$ such that $(\pi_{n,m}(f_{k}))_{k\in\N}$ has a convergent 
subsequence in $\mathcal{O}\nu_{m}^{-1}(\overline{U_{m}(K)})$ 
for every sequence $(f_{k})_{k\in\N}$ in $B_{n}$. 
According to $(qV_{\infty})$, we choose $m:=\mathcal{I}_{1}(n)>n$. Let $\varepsilon>0$. Then there is a compact set 
$Q\subset \overline{U_{m}(K)}$ with
\begin{equation}\label{satz2.1.2}
 \sup_{z \in \overline{U_{m}(K)}\setminus Q}\frac{\nu_{m}(z)^{-1}}{\nu_{n}(z)^{-1}}
=\sup_{z \in \overline{U_{m}(K)}\setminus Q}\frac{\nu_{n}(z)}{\nu_{m}(z)}
\leq \varepsilon.
\end{equation}
In addition, we set $C(n,\varepsilon):=\sup_{z\in Q}\nu_{m}(z)^{-1}>0$. 
Now, let $(f_{k})_{k \in \N}$ be a sequence in $B_{n}$. 
By $(iii)$ it has a convergent subsequence $(f_{k_{l}})_{l\in\N}$ w.r.t.\ 
system of seminorms $\|\cdot\|_{M}$, $M\subset U_{n}(K)$ compact. 
Then there exists $N\in\N$ such that for $l,j\geq N$
\begin{equation}\label{satz2.1.3}
\|f_{k_{l}}-f_{k_{j}}\|_{Q}=\sup_{z \in Q}|f_{k_{l}}(z)-f_{k_{j}}(z)| < \frac{\varepsilon}{C(n,\varepsilon)}
\end{equation}
and therefore
\begin{flalign*}
&\hspace{0.37cm}\|\pi_{n,m}(f_{k_{l}})-\pi_{n,m}(f_{k_{j}})\|_{m}\\
&\leq \sup_{z \in Q}|f_{k_{l}}(z)-f_{k_{j}}(z)|\nu_{m}(z)^{-1}
     +\sup_{z \in \overline{U_{m}(K)}\setminus Q}|f_{k_{l}}(z)-f_{k_{j}}(z)|\nu_{m}(z)^{-1}\\ 
&\leq C(n,\varepsilon)\|f_{k_{l}}-f_{k_{j}}\|_{Q}+\sup_{z \in \overline{U_{m}(K)}\setminus Q}|f_{k_{l}}(z)-f_{k_{j}}(z)|\nu_{n}(z)^{-1}
      \frac{\nu_{m}(z)^{-1}}{\nu_{n}(z)^{-1}}\\ 
&\underset{\mathclap{\eqref{satz2.1.3},\eqref{satz2.1.2}}}{\leq} C(n,\varepsilon)\frac{\varepsilon}{C(n,\varepsilon)}
 +\varepsilon \sup_{z \in \overline{U_{m}(K)}\setminus Q}|f_{k_{l}}(z)-f_{k_{j}}(z)|\nu_{n}(z)^{-1}\\
&\leq\varepsilon +\varepsilon\sup_{z \in \overline{U_{n}(K)}}|f_{k_{l}}(z)-f_{k_{j}}(z)|\nu_{n}(z)^{-1}\\
&\leq\varepsilon +\varepsilon(\|f_{k_{l}}\|_{n}+\|f_{k_{j}}\|_{n})\\
&\underset{\mathclap{f_{k} \in  B_{n}}}{\leq}3\varepsilon.
\end{flalign*}
Hence the subsequence $(\pi_{n,m}(f_{k_{l}}))_{l\in\N}$ converges in $\mathcal{O}\nu_{m}^{-1}(\overline{U_{m}(K)})$,
proving the compactness of $\pi_{n,m}$. 

It follows from $(i)-(iv)$ and \cite[Proposition 25.20, p.\ 304]{meisevogt1997} 
that the inductive limit $\mathcal{OV}^{-1}_{\operatorname{ind}}(K)$ exists and is a DFS-space.

$b)$ We set $F:=\operatorname{span}\{\delta_{x_{0}}^{(n)}\;|\;x_{0}\in K\cap\R,\,n\in\N_{0}\}$. 
Let $x_{0}\in K\cap\R$ and $n\in\N_{0}$. Then $\delta_{x_{0}}^{(n)}$ is linear and for $k\in\N$ and 
$f\in\mathcal{O}\nu_{k}^{-1}(\overline{U_{k}(K)})$ we derive from Cauchy's inequality that
\[
  |\delta_{x_{0}}^{(n)}(f)| 
\leq n!(2k)^{n}\max_{|z-x_{0}|=1/(2k)}|f(z)|
\leq n!(2k)^{n}\max_{|z-x_{0}|=1/(2k)}\nu_{k}(z)\|f\|_{k}.
\]
Hence $\delta_{x_{0}}^{(n)}$ is continuous on $\mathcal{O}\nu_{k}^{-1}(\overline{U_{k}(K)})$ for any $k\in\N$, 
implying $F\subset\mathcal{OV}^{-1}_{\operatorname{ind}}(K)'$. 
As $\mathcal{OV}^{-1}_{\operatorname{ind}}(K)$ is a DFS-space by part $a)$, 
it is reflexive by \cite[Proposition 25.19, p.\ 303]{meisevogt1997}, i.e.\ the canonical embedding
$J\colon\mathcal{OV}^{-1}_{\operatorname{ind}}(K)\to(\mathcal{OV}^{-1}_{\operatorname{ind}}(K)_{b}')_{b}'$ 
is a topological isomorphism. We consider the polar set of $F$, i.e.\
\[
F^{\circ}:=\{y\in(\mathcal{OV}^{-1}_{\operatorname{ind}}(K)_{b}')_{b}' \;|\;\forall\;T\in F:\; y(T)=0\}.
\]
Let $y\in F^{\circ}$. Then there is $f\in\mathcal{OV}^{-1}_{\operatorname{ind}}(K)$ such that 
$y=J(f)$. For $T:=\delta_{x_{0}}^{(n)}\in F$
\[
0=y(T)=J(f)(T)=T(f)=f^{(n)}(x_{0})
\]
is valid for any $n\in\N_{0}$. Thus $f$ is identical to zero on a neighbourhood of $x_{0}$ (by Taylor series expansion) 
since $f$ is holomorphic near $x_{0}\in U_{n}(K)$. Due to the assumptions every component of $U_{n}(K)$ contains a point 
$x_{0}\in K\cap\R$ so $f$ is identical to zero on $\overline{U_{n}(K)}$ by the identity theorem and continuity,
yielding to $y=0$. Therefore $F^{\circ}=\{0\}$ and thus $F$ is dense in $\mathcal{OV}^{-1}_{\operatorname{ind}}(K)_{b}'$ 
by the bipolar theorem.

$c)$ The proof is similar to $b)$. We define $F:=\operatorname{span}\{\delta_{x_{0}}\;|\;x_{0}\in K\cap\R\}$. 
Then, like above, for $y\in F^{\circ}$ there is $f\in\mathcal{OV}^{-1}_{\operatorname{ind}}(K)$ with $y=J(f)$ 
such that for $T:=\delta_{x_{0}}\in F$, $x_{0}\in K\cap\R$,
\[
0=y(T)=J(f)(T)=T(f)=f(x_{0}).
\]
Due to the assumptions every component $Z$ of $U_{n}(K)$ contains a point $x_{0}\in K\cap\R$ and every point in $Z\cap K\cap\R$ 
is an accumulation point of $Z\cap K\cap\R$. So $f$ is identical to zero on $U_{n}(K)$ by the identity theorem.
\end{proof}

The parts $(iii)$-$(iv)$ of the proof of \prettyref{prop:DFS} a) are just slight modifications of 
\cite[Theorem (b), p.\ 67-68]{Bierstedt1975} 
which cannot be directly applied due to the closure $\overline{U_{n}(K)}$ being involved. 
In the case $\nu_{n}(z)^{-1}:=\exp((1/n)|\re(z)|)$, $z\in\C$, for all $n\in\N$ the spaces $\mathcal{OV}^{-1}_{\operatorname{ind}}(K)$ 
play an essential in the theory of Fourier hyperfunctions and it is already mentioned in \cite[p.\ 469]{Kawai} resp.\ 
proved in \cite[1.11 Satz, p.\ 11]{J} and \cite[3.5 Theorem, p.\ 17]{ich} that they are DFS-spaces. 

\begin{rem}
Let $K\subset \overline{\R}$ be a non-empty compact set and $\mathcal{V}:=(\nu_{n})_{n\in\N}$ a directed family 
of continuous weights on $\C$. 
\begin{enumerate}
\item [a)] If $K\subset\R$, then $(qV_{\infty})$ is fulfilled, which follows from the choices 
$\mathcal{I}_{1}(n):=2n$ and $Q:=\overline{U_{\mathcal{I}_{1}(n)}(K)}$ for $n\in\N$.
\item [b)] If $K\subset\R$, then $\mathcal{OV}^{-1}_{\operatorname{ind}}(K)\cong\mathscr{A}(K)$.
\end{enumerate}
\end{rem}

Now, we take a closer look at the sets $U_{t}(K)$ (c.f.\ \cite[3.3 Remark, p.\ 13]{ich}).

\begin{rem}\label{rem:fin_many_comp+path}
Let $K\subset\overline{\R}$ be compact and $t\in\R$, $t\geq 1$.
\begin{enumerate}
\item [a)]
The set $U_{t}(K)$ has finitely many components.
\item [b)]
Let $K\neq\varnothing$ and $Z$ be a component of $U_{t}(K)$. We define $a:=\min(Z\cap K)$ and $b:=\max(Z\cap K)$ if existing (in $\R$).
\begin{enumerate}
\item [(i)]
If $Z$ is bounded, there exists $0<R\leq 1/t$ such that for all $0<r\leq R$: $\{z\in\C\;|\;\d(\{z\},[a,b])< r\}\subset Z$
\item [(ii)]
If $Z\cap\R$ is bounded from below and unbounded from above and $a$ exists, there exists $0<R\leq 1/t$ such that for all $0<r\leq R$: 
$\{z\in\C\;|\;\d(\{z\},[a,\infty))< r\}\subset Z$
\item [(iii)]
If $Z\cap\R$ is bounded from above and unbounded from below and $b$ exists, there exists $0<R\leq 1/t$ such that for all $0<r\leq R$: 
$\{z\in\C\;|\;\d(\{z\},(-\infty,b])< r\}\subset Z$
\item [(iv)]
If $Z\cap\R$ is unbounded from below and above, there exists $0<R\leq 1/t$ such that for all $0<r\leq R$: 
$\{z\in\C\;|\;\d(\{z\},\R)< r\}\subset Z$
\item [(v)]
If $Z\cap\R$ is bounded from below and unbounded from above and $a$ does not exist, then $Z=(t,\infty)+i (-1/t,1/t)$. 
If $Z\cap\R$ is bounded from above and unbounded from below and $b$ does not exist, then $Z=(-\infty,-t)+i (-1/t,1/t)$.
\end{enumerate}
\end{enumerate}
\end{rem}
\begin{proof}
a) We only consider the case $\infty \in K$, $-\infty \notin K$. 
Let $(Z_{j})_{j\in J}$ denote the (pairwise disjoint) components of $U_{t}(K)$. Then $U_{t}(K) = \bigcup_{j\in J}Z_{j}$ 
and by definition of a component there is $k\in J$ such that $Z_{k}$ is the only component including $(t,\infty)+i(-1/t,1/t)$. 
Furthermore there exists $m\in\R$ with $\bigcup_{j\in J\setminus\{k\}}(Z_{j}\cap\R) \subset[m,t]$ by assumption. 
For $j\neq k$ the length $\lambda(Z_{j}\cap\R)$ of the interval $Z_{j}\cap \R$, where $\lambda$ denotes the Lebesgue measure, 
is estimated from below by $\lambda(Z_{j}\cap\R)\geq 2/t$ by definition of $U_{t}(K)$. 
Since all $Z_{j}$ are pairwise disjoint, this implies that $J$ has to be finite. The others cases follow analogously.

b)(i) Since $Z\cap K$ is closed in $\R$ and therefore compact, $a$ and $b$ exist. Hence $[a,b]\subset Z$ 
by the definition of $U_{t}(K)$ and as $Z$ is connected. $[a,b]$ being a compact subset of the open set $Z$ 
implies that there is $0<R< 1/t$ such that $([a,b]+i(-R,R))\subset Z$ by the tube lemma, 
which completes the proof. 

(ii) If $Z\cap K\cap(-\infty,t]\neq\varnothing$, then $a$ exists and analogously to (i) 
there exists $0<R<1/t$ such that for all $0<r\leq R$
\[
\{z\in\C\;|\;\d(\{z\},[a,t])<r\}\subset Z.
\]
By definition of $U_{t}(K)$ this brings forth $\{z\in\C\;|\;\d(\{z\},[a,\infty))<r\}\subset Z$.
If $Z\cap K\cap(-\infty,t]=\varnothing$ and $a$ exists, the desired $0<R<1/t$ exists by the definition 
of $U_{t}(K)$ since $t\not\in Z\cap K$ and $Z\cap K$ is closed in $\R$, which implies $\d(\{t\},Z\cap K)>0$.

(iii) Analogously to (ii).

(iv) By the assumptions $Z\cap K\cap[-t,t]\neq\varnothing$. Analogously to (i) there exists $0<R<1/t$ such that for all $0<r\leq R$
\[
\{z\in\C\;|\;\d(\{z\},[-t,t])< r\}\subset Z.
\]
Like in (ii) and (iii) this brings forth $\{z\in\C\;|\;\d(\{z\},\R)<r\}\subset Z$.

(v) This follows directly from the definition of $U_{t}(K)$ and as $Z$ is a component of $U_{t}(K)$.
\end{proof}

\begin{defn}\label{def:duality_path}
Let $n\in\N$, $K\subset\overline{\R}$ be a non-empty compact set and $(Z_{j})_{j\in J}$ denote the components of $U_{n}(K)$. 
A component $Z_{j}$ of $U_{n}(K)$ fulfils one of the cases of \prettyref{rem:fin_many_comp+path} b) 
and so for $a=a_{j}$, $b=b_{j}$ (in the cases (i)-(iii)), for $0<r_{j}<R_{j}=R$ (in the cases (i)-(iv)) 
resp.\ $0<r_{j}<1/n=:R_{j}$ (in the case (v)) we define
\[
V_{r_{j}}(Z_{j}):=
\begin{cases}
\{z\in\C\;|\;\d(\{z\},[a_{j},b_{j}])<r_{j}\} &, Z_{j}\;\text{fulfils (i)},\\
\{z\in\C\;|\;\d(\{z\},[a_{j},\infty))<r_{j}\} &, Z_{j}\;\text{fulfils (ii)},\\
\{z\in\C\;|\;\d(\{z\},(-\infty,b_{j}])<r_{j}\} &, Z_{j}\;\text{fulfils (iii)},\\
\{z\in\C\;|\;\d(\{z\},\R)<r_{j}\} &, Z_{j}\;\text{fulfils (iv)}, \\
(1/r_{j},\infty)+i(-r_{j}, r_{j}) &, Z_{j}=(n,\infty)+i(-1/n,1/n),\\
(-\infty,-1/r_{j})+i(-r_{j}, r_{j}) &, Z_{j}=(-\infty, -n)+i(-1/n,1/n),
\end{cases}
\]
where $Z_{j}$ fulfils (v) in the last two cases. By \prettyref{rem:fin_many_comp+path} a) 
there is w.l.o.g.\ $k\in\N$ with $U_{n}(K)=\bigcup_{j=1}^{k}Z_{j}$. 
We set $r:=(r_{j})_{1\leq j\leq k}$ and the path
\[
 \gamma_{K,n,r}:=\sum_{j=1}^{k}\gamma_{j}
\]
where $\gamma_{j}$ is the path along the boundary of $V_{r_{j}}(Z_{j})$ in $\C$ in the positive sense (counterclockwise). 
\end{defn}
\begin{center}
 \includegraphics[scale=0.85]{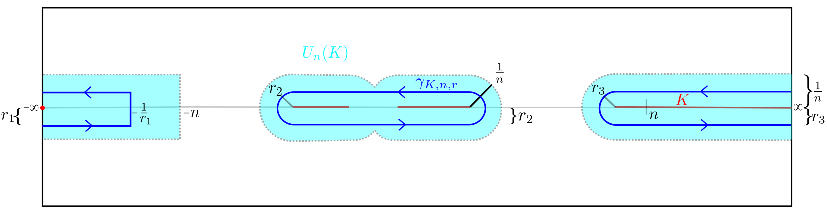}
 \captionsetup{type=figure}
 \caption{Path $\gamma_{K,n,r}$ for $\pm\infty\in K$ (c.f.\ \cite[Figure 4.1, p.\ 40]{ich})}
\end{center}
 
\begin{prop}\label{prop:pettis-integral}
Let $K\subset\overline{\R}$ be a non-empty compact set and $\mathcal{V}:=(\nu_{n})_{n\in\N}$ a directed family 
of continuous weights on $\C$ which fulfils $(qV_{\infty})$ and $(qL^{1})$. 
Let $n\in\N$, $\gamma_{K,n,r}$ be the path from \prettyref{def:duality_path} and $E$ a locally convex Hausdorff space.
If 
\begin{enumerate}
 \item [(i)] $K\subset\R$ and $E$ is locally complete, or
 \item [(ii)] $E$ is sequentially complete,
\end{enumerate}
then 
\begin{enumerate}
 \item [a)] $F\cdot\varphi$ is Pettis-integrable along $\gamma_{K,n,r}$ for all $F\in\mathcal{OV}(\overline{\C}\setminus K,E)$ 
 and $\varphi\in\mathcal{O}\nu_{n}^{-1}(\overline{U_{n}(K)})$. 
 \item [b)] there are $m\in\N$ and $C>0$ such that for all $\alpha\in\mathfrak{A}$, $F\in\mathcal{OV}(\overline{\C}\setminus K,E)$ 
 and $\varphi\in\mathcal{O}\nu_{n}^{-1}(\overline{U_{n}(K)})$
 \[
      p_{\alpha}\bigl(\int_{\gamma_{K,n,r}}F(\zeta)\varphi(\zeta)\d\zeta\bigr)
 \leq C |F|_{m,\alpha}\|\varphi\|_{n}.
 \]
 \item [c)] for all $F\in\mathcal{OV}(\overline{\C}\setminus K,E)$, $\varphi\in\mathcal{O}\nu_{n}^{-1}(\overline{U_{n}(K)})$ 
 and $\widetilde{r}:=(\widetilde{r}_{j})_{1\leq j\leq k}$ with $0<\widetilde{r}_{j}<R_{j}$ for all $1\leq j\leq k$
 \[
  \int_{\gamma_{K,n,r}}F(\zeta)\varphi(\zeta)\d\zeta=\int_{\gamma_{K,n,\widetilde{r}}}F(\zeta)\varphi(\zeta)\d\zeta .
 \]
 \item [d)] for all $F\in\mathcal{OV}(\overline{\C},E)$ and $\varphi\in\mathcal{O}\nu_{n}^{-1}(\overline{U_{n}(K)})$ 
 \[
  \int_{\gamma_{K,n,r}}F(\zeta)\varphi(\zeta)\d\zeta=0.
 \]
\end{enumerate}
\end{prop}
\begin{proof}
$a)+b)$ We have to show that there is $e_{K,n,r}\in E$ such that 
 \[
  \langle e',e_{K,n,r}\rangle=\int_{\gamma_{K,n,r}}\langle e',F(\zeta)\varphi(\zeta)\rangle\d\zeta,\quad e'\in E',
 \]
which gives $\int_{\gamma_{K,n,r}}F(\zeta)\varphi(\zeta)\d\zeta=e_{K,n,r}$.
 
First, let $V_{r_{j}}(Z_{j})$ be bounded for some $1\leq j\leq k$. Then
there is a parametrisation $\gamma_{j}\colon [0,1]\to\C$ which has a continuously differentiable extension 
$\widetilde{\gamma}_{j}$ on $(-1,2)$.  
As the map $(e'\circ (F\cdot\varphi)\circ\gamma_{j})\cdot\gamma_{j}'$ is continuous on $[0,1]$ for every $e'\in E'$, 
it is an element of $L^{1}([0,1])$ for every $e'\in E'$. Thus the map 
 \[
  \mathfrak{I}_{j}\colon E'\to \C,\;\mathfrak{I}_{j}(e'):=\int_{\gamma_{j}}\langle e',F(\zeta)\varphi(\zeta)\rangle\d\zeta
  =\int_{0}^{1}\langle e',(F\cdot\varphi)(\gamma_{j}(t))\rangle\gamma_{j}'(t)\d t,
 \]
is well-defined and linear. We estimate
 \[
  |\mathfrak{I}_{j}(e')|\leq \underbrace{\int_{0}^{1}|\gamma_{j}'(t)|\d t}_{=:\ell(\gamma_{j})} 
  \sup_{z\in (F\cdot\varphi)(\gamma_{j}([0,1]))}|e'(z)|,\quad e'\in E'.
 \]
Let us denote by $\overline{\operatorname{acx}}((F\cdot\varphi)(\gamma_{j}([0,1])))$ the closure of the 
absolutely convex hull of the set $(F\cdot\varphi)(\gamma_{j}([0,1]))$.
Since $e'\circ (F\cdot\varphi\circ\widetilde{\gamma}_{j})\in \mathcal{C}^{1}((-1,2))$ for every $e'\in E'$, 
the absolutely convex set $\overline{\operatorname{acx}}((F\cdot\varphi)(\gamma_{j}([0,1])))$ is compact in 
the locally complete space $E$ by \cite[Proposition 2, p.\ 354]{Bonet2002}, 
yielding $\mathfrak{I}_{j}\in(E_{\kappa}')'\cong E$ by the theorem of Mackey-Arens, i.e.\  
there is $e_{j}\in E$ such that
 \[
  \langle e',e_{j}\rangle=\mathfrak{I}_{j}(e')=\int_{\gamma_{j}}\langle e',F(\zeta)\varphi(\zeta)\rangle\d\zeta,\quad e'\in E'.
 \]
Therefore $F\cdot\varphi$ is Pettis-integrable along $\gamma_{j}$. Furthermore, 
we choose $m_{j}\in\N$ such that $(1/m_{j})<r_{j}$ and
for $\alpha\in\mathfrak{A}$ we set $B_{\alpha}:=\{x\in E\;|\;p_{\alpha}(x)<1\}$. We note that
\begin{align*}
   p_{\alpha}\bigl(\int_{\gamma_{j}}F(\zeta)\varphi(\zeta)\d\zeta\bigr)
 &=\sup_{e'\in B_{\alpha}^{\circ}}\bigl|\langle e',\int_{\gamma_{j}}F(\zeta)\varphi(\zeta)\d\zeta\rangle\bigr|\\
 &\leq\ell(\gamma_{j})\sup_{e'\in B_{\alpha}^{\circ}}\sup_{z\in\gamma_{j}([0,1])}|e'(F(z))\varphi(z)|\\
 &\leq \ell(\gamma_{j})\sup_{w\in\gamma_{j}([0,1])}\frac{\nu_{n}(w)}{\nu_{m_{j}}(w)}
 \sup_{e'\in B_{\alpha}^{\circ}}\sup_{z\in S_{m_{j}}(K)}|e'(F(z)\nu_{m_{j}}(z))|\|\varphi\|_{n}\\
 &=\ell(\gamma_{j})\bigl\|\frac{\nu_{n}}{\nu_{m_{j}}}\bigr\|_{\gamma_{j}([0,1])}\sup_{z\in S_{m_{j}}(K)}p_{\alpha}(F(z)\nu_{m_{j}}(z))
 \|\varphi\|_{n}\\
 &=\ell(\gamma_{j})\bigl\|\frac{\nu_{n}}{\nu_{m_{j}}}\bigr\|_{\gamma_{j}([0,1])} |F|_{m_{j},\alpha}\|\varphi\|_{n}
\end{align*}
where we used \cite[Proposition 22.14, p.\ 256]{meisevogt1997} in the first and second to last equation 
to get from $p_{\alpha}$ to $\sup_{e'\in B_{\alpha}^{\circ}}$ and back. If $K\subset\R$, then all 
$V_{r_{j}}(Z_{j})$, $1\leq j\leq k$, are bounded and we deduce
our statement with $e_{K,n,r}:=\sum_{j=1}^{k}e_{j}$, $m:=\max_{1\leq j\leq k} m_{j}$ and 
$C:=k\max_{1\leq j\leq k}\ell(\gamma_{j})\|\nu_{n}/\nu_{m_{j}}\|_{\gamma_{j}([0,1])}$. 

Second, let us consider the case $\infty\in K$, $-\infty\not\in K$. Let $Z_{k}$ be the unique unbounded component of $U_{n}(K)$. 
For $q\in\N$, $q>1/r_{k}>n$, we denote by $\gamma_{k,q}$ the part of $\gamma_{k}$ in $\{z\in\C\;|\;\re(z)\leq q\}$. 
Like in the first part the Pettis-integral 
\[
 e_{k,q}:=\int_{\gamma_{k,q}}F(\zeta)\varphi(\zeta)\d\zeta
\]
exists (in $E$) and for $\alpha\in\mathfrak{A}$ and $m_{k}\in\N$, $(1/m_{k})<r_{k}$, we have
\[
  p_{\alpha}\bigl(\int_{\gamma_{k,q}}F(\zeta)\varphi(\zeta)\d\zeta\bigr)
 \leq\ell(\gamma_{k,q})\bigl\|\frac{\nu_{n}}{\nu_{m_{k}}}\bigr\|_{\gamma_{k,q}([0,1])} |F|_{m_{k},\alpha}\|\varphi\|_{n}.
\]
Next, we prove that $(e_{k,q})_{q>1/r_{k}}$ is a Cauchy sequence in $E$. 
We choose $M:=\max(m_{k},\mathcal{I}_{2}(n))$ with $\mathcal{I}_{2}(n)$ 
from condition $(qL^{1})$. For $q,p\in\N$, $q>p>1/r_{k}>n$, we obtain
\begin{flalign*}
 &\hspace{0.37cm} p_{\alpha}(e_{k,q}-e_{k,p})\\
 &=\sup_{e'\in B_{\alpha}^{\circ}}\bigl|\int_{\gamma_{k,q}-\gamma_{k,p}}e'(F(\zeta))\varphi(\zeta)\d\zeta\bigr|\\ 
 &\leq\sup_{e'\in B_{\alpha}^{\circ}}\bigl(\int_{p}^{q}|e'(F(t-ir_{k}))\varphi(t-ir_{k})|\d t
  +\int_{p}^{q}|e'(F(t+ir_{k}))\varphi(t+ir_{k})|\d t\bigr)\\
 &\leq\sup_{e'\in B_{\alpha}^{\circ}}
  \bigl(\int_{p}^{q}\frac{\nu_{n}(t-ir_{k})}{\nu_{M}(t-ir_{k})}\d t+\int_{p}^{q}\frac{\nu_{n}(t+ir_{k})}{\nu_{M}(t+ir_{k})}\d t\bigr)
  |e'\circ F|_{M}\|\varphi\|_{n}\\
 &=2\int_{p}^{q}\frac{\nu_{n}(t)}{\nu_{M}(t)}\d t |F|_{M,\alpha}\|\varphi\|_{n}
 \leq 2\int_{p}^{q}\frac{\nu_{n}(t)}{\nu_{\mathcal{I}_{2}(n)}(t)}\d t |F|_{M,\alpha}\|\varphi\|_{n}
\end{flalign*}
and observe that $(\int_{0}^{q}\tfrac{\nu_{n}(t)}{\nu_{\mathcal{I}_{2}(n)}(t)}\d t)_{q}$ 
is a Cauchy sequence in $\C$ by condition $(qL^{1})$. 
Therefore $(e_{k,q})_{q>1/r_{k}}$ is a Cauchy sequence in $E$, has a limit $e_{k}$ in the sequentially complete space $E$ and
\[
 e_{k}=\int_{\gamma_{k}}F(\zeta)\varphi(\zeta)\d\zeta.
\]
We fix $p\in\N$, $p>1/r_{k}>n$, and conclude that
\begin{align*}
 p_{\alpha}\bigl(\int_{\gamma_{k}}F(\zeta)\varphi(\zeta)\d\zeta\bigr)
&\leq p_{\alpha}(e_{k}-e_{k,p})+p_{\alpha}(e_{k,p})\\
&\leq \bigl(2\int_{0}^{\infty}\frac{\nu_{n}(t)}{\nu_{\mathcal{I}_{2}(n)}(t)}\d t
 +\ell(\gamma_{k,p})\bigl\|\frac{\nu_{n}}{\nu_{m_{k}}}\bigr\|_{\gamma_{k,p}([0,1])}\bigr)
  |F|_{M,\alpha}\|\varphi\|_{n}
\end{align*}
Consequently, our statement holds also in the case $\infty\in K$, $-\infty\not\in K$ and in the remaining cases it follows analogously.

$c)$ We note that
\[
\langle e',\int_{\gamma_{K,n,r}}F(\zeta)\varphi(\zeta)\d\zeta
 -\int_{\gamma_{K,n,\widetilde{r}}}F(\zeta)\varphi(\zeta)\d\zeta\rangle\\
=\int_{\gamma_{K,n,r}-\gamma_{K,n,\widetilde{r}}}\langle e',F(\zeta)\varphi(\zeta)\rangle\d\zeta
\]
for all $e'\in E'$. Thus statement $c)$ follows from Cauchy's integral theorem and the Hahn-Banach theorem if $K\subset\R$. 
Now, let us consider the case $\infty\in K$, $-\infty\not\in K$. We denote by $\gamma_{k}$ resp.\ $\widetilde{\gamma}_{k}$ the part of 
$\gamma_{K,n,r}$ resp.\ $\gamma_{K,n,\widetilde{r}}$ in the unbounded component of $U_{n}(K)$. It suffices to show that 
\begin{equation}\label{eq:path_independent}
 \int_{\gamma_{k}-\widetilde{\gamma}_{k}}\langle e',F(\zeta)\varphi(\zeta)\rangle\d\zeta=0,\quad e'\in E'.
\end{equation}
We choose $\mathcal{I}_{1}(n)$ from condition $(qV_{\infty})$. 
Let $\varepsilon>0$ and w.l.o.g.\ $r_{k}<\widetilde{r}_{k}$. 
Then there is a compact set $Q\subset\overline{U_{\mathcal{I}_{1}(n)}(K)}$ with 
$\nu_{n}(z)\leq\varepsilon\nu_{\mathcal{I}_{1}(n)}(z)$ for all 
$z\in\overline{U_{\mathcal{I}_{1}(n)}(K)}\setminus Q$. We choose $q\in\R$ such that $q>1/r_{k}$ 
and $q\in\overline{U_{\mathcal{I}_{1}(n)}(K)}\setminus Q$ 
and define the path $\gamma_{0,q}^{+}\colon [r_{k},\widetilde{r}_{k}]\to\C$, $\gamma_{0,q}^{+}(t):=q+it$. 
We deduce that for $m_{k}\in\N$, $(1/m_{k})<\min(r_{k},1/\mathcal{I}_{1}(n))$, and every $e'\in E'$
\begin{align*}
      \bigl|\int_{\gamma_{0,q}^{+}}\langle e',F(\zeta)\varphi(\zeta)\rangle\d\zeta\bigr|
 &\leq \int_{r_{k}}^{\widetilde{r}_{k}}\frac{\nu_{n}(q+it)}{\nu_{m_{k}}(q+it)}\d t\|\varphi\|_{n} |e'\circ F|_{m_{k}}\\
 &=(\widetilde{r}_{k}-r_{k})\frac{\nu_{n}(q)}{\nu_{m_{k}}(q)}\|\varphi\|_{n} |e'\circ F|_{m_{k}}\\
 &\leq (\widetilde{r}_{k}-r_{k})\|\varphi\|_{n}|e'\circ F|_{m_{k}}\varepsilon 
\end{align*}
where we used condition $(qL^{1})$ for the second equality.
In the same way we obtain with $\gamma_{0,q}^{-}\colon [-\widetilde{r}_{k},-r_{k}]\to\C$, $\gamma_{0,q}^{-}(t):=q+it$, that 
\[
  \bigl|\int_{\gamma_{0,q}^{-}}\langle e',F(\zeta)\varphi(\zeta)\rangle\d\zeta\bigr|
  \leq (\widetilde{r}_{k}-r_{k})\|\varphi\|_{n}|e'\circ F|_{m_{k}}\varepsilon .
\]
Hence we get \eqref{eq:path_independent} by Cauchy's integral theorem and the Hahn-Banach theorem as well. 
The remaining cases follow similarly.

$d)$ The proof is similar to $c)$. Let $F\in\mathcal{OV}(\overline{\C},E)$. Again, it suffices to prove that 
 \[
  \int_{\gamma_{K,n,r}}\langle e',F(\zeta)\varphi(\zeta)\rangle\d\zeta=0,\quad e'\in E'.
 \]
This follows from Cauchy's integral theorem and the Hahn-Banach theorem if $K\subset\R$. 
Again, we only consider the case $\infty\in K$, $-\infty\not\in K$ and only need to show that 
\[
 \int_{\gamma_{k}}\langle e',F(\zeta)\varphi(\zeta)\rangle\d\zeta=0,\quad e'\in E',
\]
where $\gamma_{k}$ is the part of $\gamma_{K,n,r}$ in the unbounded component of $U_{n}(K)$. 
Let $\varepsilon>0$ and choose $q$ as in $c)$. 
Then we have with $\gamma_{0,q}\colon [-r_{k},r_{k}]\to\C$, $\gamma_{0,q}(t):=q+it$, that 
\[
  \bigl|\int_{\gamma_{0,q}}\langle e',F(\zeta)\varphi(\zeta)\rangle\d\zeta\bigr|
  \leq 2r_{k}\|\varphi\|_{n}|e'\circ F|_{\mathcal{I}_{1}(n)}\varepsilon 
\]
for every $e'\in E'$ by $(qV_{\infty})$ and $(qL^{1})$. 
Cauchy's integral theorem and the Hahn-Banach theorem imply our statement.
\end{proof}

An essential role in the proof of $\mathcal{O}(\C\setminus K,E)/\mathcal{O}(\C,E)\cong L_{b}(\mathscr{A}(K),E)$ for 
non-empty compact $K\subset\R$ and quasi-complete $E$ (see \eqref{eq:duality_non_weighted}) 
plays the fundamental solution $z\mapsto 1/(\pi z)$ of the Cauchy-Riemann operator. 
By the identity theorem we can consider $\mathcal{OV}(\overline{\C},E)$ as a subspace of $\mathcal{OV}(\overline{\C}\setminus K,E)$ and 
we equip the quotient space $\mathcal{OV}(\overline{\C}\setminus K,E)/\mathcal{OV}(\overline{\C},E)$ with the induced locally convex 
quotient topology (which may not be Hausdorff, see \prettyref{rem:Hausdorff_quotient}). 
We want to prove the isomorphy
\[
 \mathcal{OV}(\overline{\C}\setminus K,E)/\mathcal{OV}(\overline{\C},E) \cong L_{b}(\mathcal{OV}^{-1}_{\operatorname{ind}}(K),E)
\]
for non-empty compact $K\subset\overline{\R}$ under some assumptions on $K$, the weights $\mathcal{V}$ and the space $E$. 
Since we have to deal with functions having some growth given by the weights $\mathcal{V}$, we have to use a fundamental 
solution $z\mapsto g(z)/(\pi z)$, where $g$ is an entire function with $g(0)=1$, of the Cauchy-Riemann operator which is suitable 
for our growth conditions. 

\begin{condCT}\label{cond:duality}
Let $K\subset \overline{\R}$ be a non-empty compact set and $\mathcal{V}:=(\nu_{n})_{n\in\N}$ a directed family 
of continuous weights on $\C$.
Let there be $g_{K}\in\mathcal{O}(\C)$, $g_{K}(0)=1$, such that with $G_{K}(z):=g_{K}(z-\cdot)$, $z\in\C\setminus K$, it holds that
\begin{enumerate}
 \item [1)] for all $z_{0}\in\C\setminus K$ there is $n\in\N$ such that 
 $G_{K}(\mathbb{B}_{1/n}(z_{0}))\subset\mathcal{O}\nu_{n}^{-1}(\overline{U_{n}(K)})$ and 
 \[
 G_{K}'(z_{0}):=\bigl(\frac{\partial}{\partial z}\bigr)^{\mathcal{O}\nu_{n}^{-1}(\overline{U_{n}(K)})}G_{K}(z_{0})
 =\lim_{\substack{h\to 0\\h\in\C,\,h\neq 0}}\frac{G_{K}(z_{0}+h)-G_{K}(z_{0})}{h}
 \]
 exists in $\mathcal{O}\nu_{n}^{-1}(\overline{U_{n}(K)})$. 
 \item [2)] for all $n\in\N$ there is $\mathcal{J}_{2}(n)>n$ such that 
 \[
 |G_{K}|_{K,n,\mathcal{J}_{2}(n)}:=\sup_{z\in S_{n}(K)}\|G_{K}(z)\|_{K,\mathcal{J}_{2}(n)}\nu_{n}(z)<\infty .
 \]
 \item [3)] for all $n\in\N$ there is $\mathcal{J}_{3}(n)>n$ such that 
 \[
 |G_{K}|_{\{x_{0}\},n,\mathcal{J}_{3}(n)}:=\sup_{z\in S_{n}(\{x_{0}\})}\|G_{K}(z)\|_{\{x_{0}\},\mathcal{J}_{3}(n)}\nu_{n}(z)<\infty 
 \]
 for all $x_{0}\in K\cap\R$. 
 \item [4)] for all $n\in\N$ there are $p:=\mathcal{J}_{4}(n)$, $m:=\widetilde{\mathcal{J}}_{4}(p)$ with $2\leq p<m$ and $C>0$ such that 
 for all $z\in \overline{S_{1/n}}:=\{w\in\C\;|\; |\im(w)|\leq (1/n)\}$
 \[
  \int_{-\infty}^{\infty}\frac{|g_{K}(z-(t\pm ip))|}{\nu_{m}(t\pm ip)}\nu_{n}(z)\d t\leq C.
 \]
 \item [5)] for all $n\in\N$ and $z\in\C\setminus K$ 
 \[
 \sup_{\zeta\in S_{n}(K)}|G_{K}(z)(\zeta)|\nu_{n}^{-1}(\zeta)<\infty. 
 \]
\end{enumerate}
\end{condCT}

$(CT)$ stands for \emph{Cauchy transformation} which is the name of the inverse of the isomorphism we are searching for.

\begin{rem}
\begin{enumerate}
 \item [a)] Since $g_{K}$ is an entire function, the estimates in the conditions $(CT.2)$ and $(CT.3)$ imply that 
 $G_{K}(S_{n}(K))\subset \mathcal{O}\nu_{\mathcal{J}_{2}(n)}^{-1}(\overline{U_{\mathcal{J}_{2}(n)}(K)})$ 
 and $G_{K}(S_{n}(\{x_{0}\}))\subset \mathcal{O}\nu_{\mathcal{J}_{3}(n)}^{-1}(\overline{U_{\mathcal{J}_{3}(n)}(\{x_{0}\})})$.
 \item [b)] If $K\subset\R$, then $(CT.2)$ implies $(CT.3)$. Indeed, we choose $\mathcal{J}_{3}(n):=\mathcal{J}_{2}(n)$, observe that 
 $M:=S_{n}(\{x_{0}\})\setminus S_{n}(K)$ is a compact set as $K\subset\R$ and 
 \begin{flalign*}
 &\hspace{0.37cm}|G_{K}|_{\{x_{0}\},n,\mathcal{J}_{3}(n)}\\
 &\leq |G_{K}|_{K,n,\mathcal{J}_{2}(n)}
      +\|\nu_{\mathcal{J}_{2}(n)}^{-1}\|_{\overline{U_{n}(\{x_{0}\})}}\,\bigr\|z\mapsto \|G_{K}(z)\|_{\overline{U_{n}(\{x_{0}\})}}\,\nu_{n}(z)\bigl\|_{M}.
 \end{flalign*}
\end{enumerate}
\end{rem}

We will see that the conditions $(qV_{\infty})$, $(qL^{1})$ and $(CT)$
hold with $g_{K}(z):=\exp(-z^{2})$, $z\in\C$, 
for $\nu_{n}(z):=\exp(a_{n}|\re(z)|^{\gamma}$), $z\in\C$, where $0<\gamma\leq 1$ and $(a_{n})_{n\in\N}$ is increasing without change of sign.

\begin{prop}\label{prop:duality_forward}
Let $K\subset\overline{\R}$ be a non-empty compact set and $\mathcal{V}:=(\nu_{n})_{n\in\N}$ a directed family 
of continuous weights on $\C$ which fulfils $(qV_{\infty})$ and $(qL^{1})$, $\gamma_{K,n,r}$ the path from \prettyref{def:duality_path} 
and $E$ a locally convex Hausdorff space. 
If 
\begin{enumerate}
 \item [(i)] $K\subset\R$ and $E$ is locally complete, or
 \item [(ii)] $E$ is sequentially complete,
\end{enumerate}
then the map 
\[
 H_{K}\colon \mathcal{OV}(\overline{\C}\setminus K,E)/\mathcal{OV}(\overline{\C},E)
 \to L_{b}(\mathcal{OV}^{-1}_{\operatorname{ind}}(K),E)
\]
given by 
\[
 H_{K}(f)(\varphi):=\int_{\gamma_{K,n,r}}F(\zeta)\varphi(\zeta)\d\zeta
\]
for $f=[F]\in \mathcal{OV}(\overline{\C}\setminus K,E)/\mathcal{OV}(\overline{\C},E)$ and 
$\varphi\in \mathcal{O}\nu_{n}^{-1}(\overline{U_{n}(K)})$, $n\in\N$, 
is well-defined, linear and continuous. For all non-empty compact sets $K_{1}\subset K$ it holds that
\begin{equation}\label{unabh.H}
{H_{K}}_{\mid\mathcal{OV}(\overline{\C}\setminus K_{1},E)/\mathcal{OV}(\overline{\C},E)}=H_{K_{1}}
\end{equation}
on $\mathcal{OV}^{-1}_{\operatorname{ind}}(K)$.
\end{prop}
\begin{proof}
In the following we omit the index $K$ in $H_{K}$ if no confusion seems to be likely. 
Let $f=[F]\in \mathcal{OV}(\overline{\C}\setminus K,E)/\mathcal{OV}(\overline{\C}\,E)$ 
and $\varphi\in\mathcal{OV}^{-1}_{\operatorname{ind}}(K)$. 
Then there is $n\in\N$ such that $\varphi\in \mathcal{O}\nu_{n}^{-1}(\overline{U_{n}(K)})$. 
Due to \prettyref{prop:pettis-integral} a) and d) $H(f)(\varphi)\in E$ 
and $H(f)$ is independent of the representative $F$ of $f$.  
From \prettyref{prop:pettis-integral} c) follows that $H(f)$ is well-defined on $\mathcal{OV}^{-1}_{\operatorname{ind}}(K)$, i.e.\ 
for all $k\in\N$, $k\geq n$, and $\varphi\in \mathcal{O}\nu_{n}^{-1}(\overline{U_{n}(K)})$ it holds that
\[
 H(f)(\varphi)=H(f)(\varphi_{\mid U_{k}(K)})=H(f)(\pi_{n,k}(\varphi)).
\]
For every $n\in\N$ there are $m\in\N$ and $C>0$ such that for all 
$f=[F]\in \mathcal{OV}(\overline{\C}\setminus K,E)/\mathcal{OV}(\overline{\C}\,E)$, 
$\varphi\in\mathcal{O}\nu_{n}^{-1}(\overline{U_{n}(K)})$ and $\alpha\in\mathfrak{A}$
\begin{equation}\label{satz2.2.3}
 p_{\alpha}(H(f)(\varphi))
 \leq C |F|_{m,\alpha}\|\varphi\|_{n}
\end{equation}
by \prettyref{prop:pettis-integral} b), which implies that $H(f)\in L(\mathcal{O}\nu_{n}^{-1}(\overline{U_{n}(K)}),E)$ 
for every $n\in\N$. We deduce that $H(f)\in L(\mathcal{OV}^{-1}_{\operatorname{ind}}(K),E)$ 
by \cite[3.6 Satz, p.\ 117]{F/W/Buch}. Let
\[
 q\colon\mathcal{OV}(\overline{\C}\setminus K,E)\to\mathcal{OV}(\overline{\C}\setminus K,E)/\mathcal{OV}(\overline{\C},E),
\; q(F):=[F],
\]
denote the quotient map. We equip the quotient space with its usual quotient topology generated by 
the system of quotient seminorms given by
\[
 |f|_{l,\alpha}^{\wedge}:=\inf_{F\in q^{-1}(f)}|F|_{l,\alpha},\quad 
 f\in\mathcal{OV}(\overline{\C}\setminus K,E)/\mathcal{OV}(\overline{\C},E),
\]
for $l\in\N$ and $\alpha\in\mathfrak{A}$. 
Then the quotient space, equipped with these seminorms, is a locally convex space (but maybe not Hausdorff). 
Since \eqref{satz2.2.3} holds for every representative $F$ of $f$, we obtain for every 
$f\in\mathcal{OV}(\overline{\C}\setminus K,E)/\mathcal{OV}(\overline{\C},E)$, 
$\varphi\in\mathcal{O}\nu_{n}^{-1}(\overline{U_{n}(K)})$, $n\in\N$, and $\alpha\in\mathfrak{A}$ that 
\begin{equation}\label{satz2.2.4}
 p_{\alpha}(H(f)(\varphi))\leq C \inf_{F\in q^{-1}(f)}|F|_{m,\alpha}\|\varphi\|_{n}=C|f|_{m,\alpha}^{\wedge}\|\varphi\|_{n}.
\end{equation}
Now, let $M\subset\mathcal{OV}^{-1}_{\operatorname{ind}}(K)$ be a bounded set. 
Since the sequence $(B_{n})_{n\in\N}$ of closed unit balls $B_n$ of 
$\mathcal{O}\nu_{n}^{-1}(\overline{U_{n}(K)})$ is a fundamental system of bounded sets
in $\mathcal{OV}^{-1}_{\operatorname{ind}}(K)$ by \cite[Proposition 25.19, p.\ 303]{meisevogt1997}, 
there exist $n\in\N$ and $\lambda>0$ with $M\subset\lambda B_{n}$. We derive 
from \eqref{satz2.2.4} that
\[
 \sup_{\varphi\in M}p_{\alpha}(H(f)(\varphi))
 \leq|\lambda| C|f|_{m, \alpha}^{\wedge},
\]
proving the continuity of $H$.

Moreover, let $K_{1}\subset\overline{\R}$ be compact and $K_{1}\subset K$. 
We observe that for every $F\in\mathcal{OV}(\overline{\C}\setminus K_{1},E)$ and  
$\varphi\in\mathcal{O}\nu_{n}^{-1}(\overline{U_{n}(K)})$, $n\in\N$, it holds that
\[
 H_{K}([F])(\varphi)=\int_{\gamma_{K,n,r}}F(\zeta)\varphi(\zeta)\d\zeta
=\int_{\gamma_{K_{1},n,r}}F(\zeta)\varphi(\zeta)\d\zeta
=H_{K_{1}}([F])(\varphi)
\]
by $(qV_{\infty})$ and $(qL^{1})$ using Cauchy's integral theorem and the Hahn-Banach theorem as 
in \prettyref{prop:pettis-integral} c) and d). This yields to
\[
{H_{K}}_{\mid\mathcal{OV}(\overline{\C}\setminus K_{1},E)/\mathcal{OV}(\overline{\C},E)}=H_{K_{1}}
\]
on $\mathcal{OV}^{-1}_{\operatorname{ind}}(K)$.
\end{proof}

Now, we take a closer look at the potential inverse of $H_{K}$.

\begin{prop}\label{prop:duality_backward}
Let $K\subset\overline{\R}$ be a non-empty compact set and $\mathcal{V}:=(\nu_{n})_{n\in\N}$ a directed family 
of continuous weights on $\C$ which fulfils $(qV_{\infty})$, $(CT.1)$ and $(CT.2)$, 
and $E$ be a locally convex Hausdorff space. 
Then the map 
\[
 \Theta_{K}\colon  L_{b}(\mathcal{OV}^{-1}_{\operatorname{ind}}(K),E)
 \to\mathcal{OV}(\overline{\C}\setminus K,E)/\mathcal{OV}(\overline{\C},E)
\]
given by 
\[
   \Theta_{K}(T)
 :=\bigl[\C\setminus K \ni z\longmapsto \frac{1}{2\pi i}\langle T,\frac{g_{K}(z-\cdot)}{z-\cdot}\rangle\bigr],
 \quad T\in L_{b}(\mathcal{OV}^{-1}_{\operatorname{ind}}(K),E), 
\] 
is well-defined, linear and continuous. 
\end{prop}
\begin{proof}
We omit the index $K$ in $g_{K}$ and $G_{K}$ from condition $(CT)$. Due to condition $(qV_{\infty})$ 
the inductive limit $\mathcal{OV}^{-1}_{\operatorname{ind}}(K)$ exists by \prettyref{prop:DFS} a). 
For $z\in\C\setminus K$ and $\zeta\in\C\setminus\{z\}$ we define 
\[
 \widetilde{g}(z,\zeta):=\frac{G(z)(\zeta)}{z-\zeta}=\frac{g(z-\zeta)}{z-\zeta}
\]
and note that $\widetilde{g}(z,\cdot)\in\mathcal{O}(\C\setminus\{z\})$. 
Let $z\in\C\setminus K$. Then there is $n\in\N$ such that 
$G(\mathbb{B}_{1/n}(z))\subset\mathcal{O}\nu_{n}^{-1}(\overline{U_{n}(K)})$ by $(CT.1)$. 
Further, there is $k=k(z)\in\N$ such that 
\[
\d_{k}:=\d(\mathbb{B}_{1/k}(z), \overline{U_{k}(K)})>0.
\]
We set $m:=\max(k,n)$ and obtain 
\[
  \|\widetilde{g}(w,\cdot)\|_{m}
  =\sup_{\zeta\in\overline{U_{m}(K)}}\frac{1}{|w-\zeta|}|G(w)(\zeta)|\nu_{m}(\zeta)^{-1}
 \leq \frac{1}{\d^{|\cdot|}(\{w\}, \overline{U_{k}(K)})}\|G(w)\|_{n}<\infty
\]
for all $w\in\mathbb{B}_{1/m}(z)$. 
We deduce that $\widetilde{g}(w,\cdot)\in\mathcal{O}\nu_{m}^{-1}(\overline{U_{m}(K)})$ for all 
$w\in\mathbb{B}_{1/m}(z)$.
We note that $G'(z)(\zeta)=g^{(1)}(z-\zeta)$ for all $\zeta\in\overline{U_{n}(K)}$ since the topology of 
$\mathcal{O}\nu_{n}^{-1}(\overline{U_{n}(K)})$ is stronger than the topology of pointwise convergence, 
which implies that
\[
\frac{\partial}{\partial z}\widetilde{g}(z,\zeta)=\frac{g^{(1)}(z-\zeta)}{z-\zeta}-\frac{g(z-\zeta)}{(z-\zeta)^{2}}
=\frac{G'(z)(\zeta)}{z-\zeta}-\frac{G(z)(\zeta)}{(z-\zeta)^{2}}
\]
for all $\zeta\in\overline{U_{n}(K)}$. 
Let $h\in\C$ with $0<|h|<1/k$. Then
\[
 \bigl|\frac{1}{z+h-\zeta}-\frac{1}{z-\zeta}\bigr|=\bigl|\frac{-h}{(z+h-\zeta)(z-\zeta)}\bigr|
 \leq \frac{|h|}{\d_{k}^2}
\]
and 
\[
  \bigl|\frac{1}{h}\bigl(\frac{1}{z+h-\zeta}-\frac{1}{z-\zeta}\bigr)+\frac{1}{(z-\zeta)^{2}}\bigr|
 = \bigl|\frac{h}{(z+h-\zeta)(z-\zeta)^{2}}\bigr|
 \leq \frac{|h|}{\d_{k}^3}
\]
for all $\zeta\in\overline{U_{k}(K)}$. It follows that
\begin{flalign*}
&\hspace{0.37cm} \bigl|\frac{\widetilde{g}(z+h,\zeta)-\widetilde{g}(z,\zeta)}{h}
 -\frac{\partial}{\partial z}\widetilde{g}(z,\zeta)\bigr|\\
&\leq\phantom{+} \frac{1}{|z+h-\zeta|}\bigl|\frac{G(z+h)(\zeta)-G(z)(\zeta)}{h}-G'(z)(\zeta)\bigr|+|G'(z)(\zeta)|\bigl|\frac{1}{z+h-\zeta}-\frac{1}{z-\zeta}\bigr| \\
&\phantom{\leq} + |G(z)(\zeta)| \bigl|\frac{1}{h}\bigl(\frac{1}{z+h-\zeta}-\frac{1}{z-\zeta}\bigr)+\frac{1}{(z-\zeta)^{2}}\bigr|
\end{flalign*}
for all $\zeta\in\overline{U_{m}(K)}$, which implies 
\begin{flalign*}
&\hspace{0.37cm} \|\frac{\widetilde{g}(z+h,\cdot)-\widetilde{g}(z,\cdot)}{h}-\frac{\partial}{\partial z}\widetilde{g}(z,\cdot)\|_{m}\\
&\leq \frac{1}{\d_{k}}\bigl\|\frac{G(z+h)-G(z)}{h}-G'(z)\bigr\|_{n}+\|G'(z)\|_{n}\frac{|h|}{\d_{k}^2}+\|G(z)\|_ {n}\frac{|h|}{\d_{k}^3}.
\end{flalign*}
We conclude that $\tfrac{\partial}{\partial z}\widetilde{g}(z,\cdot)\in\mathcal{O}\nu_{m}^{-1}(\overline{U_{m}(K)})$ 
and $\tfrac{\widetilde{g}(z+h,\cdot)-\widetilde{g}(z,\cdot)}{h}$ converges to 
$\tfrac{\partial}{\partial z}\widetilde{g}(z,\cdot)$ in $\mathcal{O}\nu_{m}^{-1}(\overline{U_{m}(K)})$ 
as $h\to 0$ by $(CT.1)$. 
Hence for all $T\in L(\mathcal{OV}^{-1}_{\operatorname{ind}}(K),E)$ the limit
\begin{align*}
  \lim_{\substack{h\to 0\\h\in\C,\,h\neq 0}}\frac{\langle T,\widetilde{g}(z+h,\cdot)\rangle-\langle T,\widetilde{g}(z,\cdot)\rangle}{h}
&=\langle T,\lim_{\substack{h\to 0\\h\in\C,\,h\neq 0}}\frac{\widetilde{g}(z+h,\cdot)-\widetilde{g}(z,\cdot)}{h}\rangle
=\langle T,\frac{\partial}{\partial z}\widetilde{g}(z,\cdot)\rangle
\end{align*}
exists in $E$, meaning that $(z\mapsto \frac{1}{2\pi i}\langle T,\frac{g(z-\cdot)}{z-\cdot}\rangle)\in\mathcal{O}(\C\setminus K,E)$. 

Let $l\in\N$. Then there is $\mathcal{J}_{2}(l)>l$ such that 
$G(S_{l}(K))\subset \mathcal{O}\nu_{\mathcal{J}_{2}(l)}^{-1}(\overline{U_{\mathcal{J}_{2}(l)}(K)})$ by $(CT.2)$. 
Moreover, there is $k\in\N$ such that 
\[
\operatorname{D}_{l,k}:=\d(S_{l}(K),\overline{U_{k}(K)})>0.
\]
Again, it follows that $\widetilde{g}(S_{l}(K), \cdot)\subset\mathcal{O}\nu_{m}^{-1}(\overline{U_{m}(K)})$ with $m:=\max(\mathcal{J}_{2}(l),k)$. 
Furthermore, we observe that $M:=\{\widetilde{g}(z,\cdot)\nu_{l}(z)\;|\;z\in S_{l}(K)\}\subset \mathcal{O}\nu_{m}^{-1}(\overline{U_{m}(K)})$ and 
\[
\sup_{\varphi\in M}\|\varphi\|_{m}=\sup_{z\in S_{l}(K)}\|\widetilde{g}(z,\cdot)\|_{m}\nu_{l}(z)
\leq \frac{1}{\operatorname{D}_{l,k}}\|G(z)\|_{K,l,\mathcal{J}_{2}(l)}<\infty
\]
by $(CT.2)$, showing that $M$ is bounded in $\mathcal{OV}^{-1}_{\operatorname{ind}}(K)$ 
by \cite[Proposition 25.19, p.\ 303]{meisevogt1997}. 
For every $\alpha\in\mathfrak{A}$ and $T\in L(\mathcal{OV}^{-1}_{\operatorname{ind}}(K),E)$ we have 
\begin{align*}
 |\Theta_{K}(T)|_{l,\alpha}^{\wedge}&\leq\bigl|z\mapsto \frac{1}{2\pi i}\langle T,\widetilde{g}(z,\cdot)\rangle\bigr|_{l,\alpha}
 =\frac{1}{2\pi}\sup_{z\in S_{l}(K)}p_{\alpha}(\langle T,\widetilde{g}(z,\cdot)\nu_{l}(z)\rangle)\\
 &=\frac{1}{2\pi}\sup_{\varphi\in M}p_{\alpha}(T(\varphi))<\infty
\end{align*}
and therefore the map 
\[
\Theta_{K}\colon L_{b}(\mathcal{OV}^{-1}_{\operatorname{ind}}(K),E)\to\mathcal{OV}(\overline{\C}\setminus K,E)/\mathcal{OV}(\overline{\C},E)
\]
is well-defined, clearly linear and continuous.
\end{proof}

The map $\Theta_{K}$ is sometimes called (weighted) Cauchy transformation for obvious reasons (see \cite[p.\ 84]{Mori2}).

\begin{thm}\label{thm:duality}
Let $K\subset\overline{\R}$ be a non-empty compact set and $\mathcal{V}:=(\nu_{n})_{n\in\N}$ a directed family 
of continuous weights on $\C$ which fulfils $(qV_{\infty})$, $(qL^{1})$ and $(CT)$, 
and $E$ a locally convex Hausdorff space. 
If 
\begin{enumerate}
 \item [(i)] $K\subset\R$ and $E$ is locally complete, or
 \item [(ii)] $K\cap\{\pm\infty\}$ has no isolated points in $K$ and $E$ is sequentially complete,
\end{enumerate}
then the map 
\[
H_{K}\colon \mathcal{OV}(\overline{\C}\setminus K,E)/\mathcal{OV}(\overline{\C},E) 
\to L_{b}(\mathcal{OV}^{-1}_{\operatorname{ind}}(K),E)
\]
is a topological isomorphism with inverse $\Theta_{K}$.
\end{thm}
\begin{proof}
As before we omit the index $K$ of $H_{K}$, $\Theta_{K}$ and $G_{K}$ if it is not necessary. 
As a consequence of \prettyref{prop:duality_forward} and \prettyref{prop:duality_backward} 
the maps $H$ and $\Theta$ are linear and continuous. First, we prove that $\Theta\circ H=\id$ on 
$\mathcal{OV}(\overline{\C}\setminus K,E)/\mathcal{OV}(\overline{\C},E)$, which implies 
the injectivity of $H$.
Let $p\in\N$, $p\geq 2$. We choose $n\in\N$ such that $\d^{|\cdot|}(S_{p}(K),\overline{U_{n}(K)})>0$. 
We define the path $\Gamma_{p}:=\Gamma_{-}-\Gamma_{+}$ with
\[
\Gamma_{\pm}\colon\R\to\C,\;\Gamma_{\pm}(t):=t\pm ip,
\]
Further, we choose $m\in\N$ such that $1/m<\min_{1\leq j\leq k}r_{j}<1/n$ and $m>p$ 
where $r=(r_{j})_{1\leq j\leq k}$ is from the path $\gamma_{K,n,r}$ in the definition of $H$. 
Due to this choice $\Gamma_{\pm}$ and $\gamma_{K,n,r}$ are within $S_{m}(K)$.

Let $f=[F]\in\mathcal{OV}(\overline{\C}\setminus K,E)/\mathcal{OV}(\overline{\C},E)$ and 
$z=x+iy\in S_{p}(K)$.
Let $u\in\R$, $u\neq x$, and $[t_{0},t_{1}]\subset[-p,p]$ such that the path $\gamma_{u}\colon[t_{0},t_{1}]\to \C$, 
$\gamma_{u}(t):=u+it$, is within $S_{m}(K)$. The map $\zeta\mapsto F(\zeta)\frac{G(z)(\zeta)}{z-\zeta}$ 
is holomorphic on $\C\setminus\{z\}$ with values in $E$ and like in \prettyref{prop:pettis-integral} a) and b) we deduce that 
it is Pettis-integrable along $\gamma_{u}$ and ${\Gamma_{\pm}}_{\mid [s_{0},s_{1}]}$ with $[s_{0},s_{1}]\subset\R$ 
using \cite[Proposition 2, p.\ 354]{Bonet2002} and the Mackey-Arens theorem.
Then we have by $(CT.5)$
\begin{align*}
\bigl|\langle e',\int_{\gamma_{u}}F(\zeta)\frac{G(z)(\zeta)}{\zeta-z}\d\zeta\rangle\bigr|
&\leq|e'\circ F|_{m}\int^{t_{1}}_{t_{0}}|G(z)(u+it)|\nu_{m}(u+it)^{-1}\frac{1}{|z-u-it|}\d t\\
&\leq|e'\circ F|_{m}(t_{1}-t_{0})\sup_{\zeta\in S_{m}(K)}|G(z)(\zeta)|\nu_{m}(\zeta)^{-1}\frac{1}{|x-u|}\\
&\underset{\mathclap{|u|\to\infty}}{\to} 0
\end{align*}
for all $e'\in E'$.
Hence we derive from Cauchy's integral formula that
\[
\langle e',F(z)\rangle=\frac{1}{2\pi i}\int_{\Gamma_{p}-\gamma_{K,n,r}}\langle e',F(\zeta)\frac{G(z)(\zeta)}{\zeta-z}\rangle\d\zeta
=-\frac{1}{2\pi i}\int_{\Gamma_{p}-\gamma_{K,n,r}}\langle e',F(\zeta)\frac{G(z)(\zeta)}{z-\zeta}\rangle\d\zeta
\]
for all $e'\in E'$ and $z\in S_{p}(K)$. Thus we have 
\[
F(z)=-\frac{1}{2\pi i}\int_{\Gamma_{p}-\gamma_{K,n,r}}F(\zeta)\frac{G(z)(\zeta)}{z-\zeta}\d\zeta
\]
for all $z\in S_{p}(K)$.
By (the proof) of \prettyref{prop:duality_backward} the function 
$\widetilde{g}(z,\cdot)=\frac{G(z)}{z-\cdot}\in\mathcal{OV}^{-1}_{\operatorname{ind}}(K)$ for all 
$z\in\C\setminus K$ and
\[
W\colon\C\setminus K\to E,\; W(z):=\frac{1}{2\pi i} H([F])(\widetilde{g}(z,\cdot))-F(z),
\]
is an element of $\mathcal{OV}(\overline{\C}\setminus K,E)$
since $T:=H([F])\in L(\mathcal{OV}^{-1}_{\operatorname{ind}}(K),E)$ by \prettyref{prop:duality_forward}. 
It follows that
\begin{align}\label{satz2.2.8}
W(z)
&= \frac{1}{2\pi i}\int_{\gamma_{K,n,r}}F(\zeta)\widetilde{g}(z,\zeta)\d\zeta
  +\frac{1}{2\pi i}\int_{\Gamma_{p}-\gamma_{K,n,r}}F(\zeta)\frac{G(z)(\zeta)}{z-\zeta}\d\zeta\notag\\
&=\frac{1}{2\pi i}\int_{\Gamma_{p}}F(\zeta)\widetilde{g}(z,\zeta)\d\zeta
=\frac{1}{2\pi i}\int_{\Gamma_{p}}F(\zeta)\frac{g(z-\zeta)}{z-\zeta}\d\zeta=:W_{p}(z)
\end{align}
for all $z\in S_{p}(K)$. 
But the right-hand side $W_{p}$ of \eqref{satz2.2.8}, as a function in $z$, is weakly holomorphic on 
$S_{p}(\varnothing)=\{z\in\C\;|\;|\im(z)|<p\}$, which follows from
\[
 \langle e',W_{p}(z)\rangle
=\langle e',\frac{1}{2\pi i}\int_{\Gamma_{p}}F(\zeta)\frac{g(z-\zeta)}{z-\zeta}\d\zeta\rangle \\
=\frac{1}{2\pi i}\int_{\Gamma_{p}}e'(F(\zeta))\frac{g(z-\zeta)}{z-\zeta}\d\zeta,\quad e'\in E',
\]
and differentiation under the integral sign. The weak holomorphy and the local completeness of $E$ imply 
that $W_{p}$ is holomorphic on $S_{p}(\varnothing)$ 
by \cite[Corollary 2, p.\ 404]{grosse-erdmann2004}. 
Thus $W$ is extended by $W_{p}$ to a function in $\mathcal{O}(\C,E)$ and
the extensions for each $p\in\N$ coincide because of the identity theorem. We denote this extension by $W$ as well.

For $l\in\N$ we choose $p:=\mathcal{J}_{4}(l)\geq 2$ and $m:=\widetilde{\mathcal{J}}_{4}(p)>p$ 
from condition $(CT.4)$. 
Then we have for $z=x+iy\in\overline{S_{1/l}}\subset S_{p}(\varnothing)$ 
\begin{flalign*}
&\hspace{0.37cm} 2\pi |\langle e',W(z)\rangle|\\
&=2\pi |\langle e',W_{p}(z)\rangle|
 =\bigl|\int_{\Gamma_{p}}e'(F(\zeta))\frac{g(z-\zeta)}{z-\zeta}\d\zeta\bigr|\\
&\leq \int_{-\infty}^{\infty}|e'(F(t-ip))|\frac{|g(z-(t-ip))|}{|z-(t-ip)|}\d t
+\int_{-\infty}^{\infty}|e'(F(t+ip))|\frac{|g(z-(t+ip))|}{|z-(t+ip)|}\d t\\
&\leq \bigl(\frac{1}{|y+p|}+\frac{1}{|y-p|}\bigr)
 \max\int_{-\infty}^{\infty}\frac{|g(z-(t\pm ip))|}{\nu_{m}(z-(t\pm ip))}\d t\,|e'\circ F|_{K,m}\\
&\leq 2\frac{1}{p-\frac{1}{l}}\max\int_{-\infty}^{\infty}\frac{|g(z-(t\pm ip))|}{\nu_{m}(z-(t\pm ip))}\d t\,
      |e'\circ F|_{K,m}
\end{flalign*}
and by $(CT.4)$ there exists $C>0$ such that for all $\alpha\in\mathfrak{A}$
\begin{align*}
 \sup_{z\in\overline{S_{1/l}}}p_{\alpha}(W(z))\nu_{l}(z)
=\sup_{e'\in B_{\alpha}^{\circ}}\sup_{\substack{0\leq |y|\leq\frac{1}{l}\\x\in\mathbb{R}}}|\langle e',W(x+iy)\rangle|\nu_{l}(x+iy)
&\leq \frac{C}{\pi(p-\frac{1}{l})}|F|_{K,m,\alpha},
\end{align*}
yielding to
\[
 |W|_{\varnothing,l,\alpha}
=\sup_{z\in S_{l}(\varnothing)}p_{\alpha}(W(z))\nu_{l}(z) 
\leq\max\bigl(|W|_{K,l,\alpha},\sup_{z\in \overline{S_{1/l}}}p_{\alpha}(W(z))\nu_{l}(z)\bigr)
<\infty .
\]
Hence $W\in\mathcal{OV}(\overline{\C},E)$ and thus
\[
(\Theta\circ H)(f)=\bigl[z\mapsto \frac{1}{2\pi i} H([F])(\widetilde{g}(z,\cdot))-F(z)\bigr]+f=[W]+f=f,
\]
i.e.\ $H$ is injective.

Second, we prove that $H\circ\Theta=\id$ on $L(\mathcal{OV}^{-1}_{\operatorname{ind}}(K),E)$, 
which implies the surjectivity of $H$.
Due to the Hahn-Banach theorem this is equivalent to the condition that
\[
e'((H\circ\Theta)(T)(\varphi))=e'(T(\varphi))
\]
holds for all $T\in L(\mathcal{OV}^{-1}_{\operatorname{ind}}(K),E)$, $\varphi\in\mathcal{OV}^{-1}_{\operatorname{ind}}(K)$ 
and $e'\in E'$.
Since
\begin{flalign*}
&\hspace{0.37cm} e'((H\circ\Theta)(T)(\varphi))\\
&=e'\bigl(H([z\mapsto \frac{1}{2\pi i} \langle T, \widetilde{g}(z,\cdot)\rangle])(\varphi)\bigr)
 =e'\bigl(\frac{1}{2\pi i}\int_{\gamma_{K,n,r}}\langle T,\widetilde{g}(z,\cdot)\rangle\varphi(z)\d z\bigr)\\
&=\frac{1}{2\pi i}\int_{\gamma_{K,n,r}}\langle e'\circ T,\widetilde{g}(z,\cdot)\rangle\varphi(z)\d z
 =(H\circ\Theta)(e'\circ T)(\varphi)
\end{flalign*}
all $T\in L(\mathcal{OV}^{-1}_{\operatorname{ind}}(K),E)$, $\varphi\in\mathcal{O}\nu_{n}^{-1}(\overline{U_{n}(K)})$, $n\in\N$, 
and $e'\in E'$, it suffices to show the result for $E=\C$.

As the span of the set of point evaluations of complex derivatives $\{\delta_{x_{0}}^{(n)}\;|\; x_{0}\in K\cap\R,\,n\in\N_{0}\}$
is dense in $\mathcal{OV}^{-1}_{\operatorname{ind}}(K)_{b}'$ by \prettyref{prop:DFS} b), we only need to show that
$(H\circ\Theta)(\delta_{x_{0}}^{(n)})(\varphi)=\langle\delta_{x_{0}}^{(n)},\varphi\rangle$ for all 
$x_{0}\in K\cap\R$, $n\in\N_{0}$ and $\varphi\in\mathcal{OV}^{-1}_{\operatorname{ind}}(K)$.
Let $x_{0}\in K\cap\R$ and $n\in\N_{0}$. Now, we have
\begin{equation}\label{satz2.2.9a}
 (H\circ\Theta)(\delta_{x_{0}}^{(n)})(\varphi)
=\frac{1}{2\pi i}\int_{\gamma_{K,k,r}}\langle\delta_{x_{0}}^{(n)},\widetilde{g}(z,\cdot)\rangle\varphi(z)\d z
\end{equation}
for all $\varphi\in\mathcal{O}\nu_{k}^{-1}(\overline{U_{k}(K)})$, $k\in\N$.
Let us take a closer look at the integral on the right-hand side of \eqref{satz2.2.9a}.
Let $m\in\N$, $m\geq 2$. 
Then $\widetilde{g}(z,\cdot)=\frac{g_{K}(z-\cdot)}{z-\cdot}\in\mathcal{O}(\mathbb{B}_{1/m}(x_{0}))$ 
for every $z\in S_{m}(\{x_{0}\})$. We fix the notation $\widetilde{g}_{z}(\zeta):=\widetilde{g}(z,\zeta)$ 
for $z\in S_{m}(\{x_{0}\})$ and $\zeta\in\mathbb{B}_{1/m}(x_{0})$. 
We set $l:=\mathcal{J}_{3}(m)>m$ with $\mathcal{J}_{3}(m)$ from condition $(CT.3)$. 
Then we get by Cauchy's inequality
\begin{align*}
     |\widetilde{g}_{z}^{\,(n)}(x_{0})|
&\leq n!l^{n}\max_{\zeta\in\partial\mathbb{B}_{1/l}(x_{0})}|\widetilde{g}(z,\zeta)|\\
&\leq n!l^{n}\frac{l-m}{ml}\max_{\zeta\in\partial\mathbb{B}_{1/l}(x_{0})}|g_{K}(z-\zeta)|
\end{align*}
for every $z\in S_{m}(\{x_{0}\})$. We deduce from $(CT.3)$ that with $G_{K}(z)=g_{K}(z-\cdot)$
\begin{flalign*}
&\hspace{0.37cm}\sup_{z\in S_{m}(\{x_{0}\})}|\widetilde{g}_{z}^{\,(n)}(x_{0})|\nu_{m}(z)\\
&\leq n!l^{n}\frac{l-m}{ml}\sup_{z\in S_{m}(\{x_{0}\})}
      \max_{\zeta\in\partial\mathbb{B}_{1/l}(x_{0})}|G_{K}(z)(\zeta)|\nu_{m}(z)\\
&\leq n!l^{n}\frac{l-m}{ml}\sup_{z\in S_{m}(\{x_{0}\})}\max_{\zeta\in\partial\mathbb{B}_{1/l}(x_{0})}
 |G_{K}(z)(\zeta)|\nu_{l}(\zeta)^{-1}\nu_{l}(\zeta)\nu_{m}(z)\\
&\leq n!l^{n}\frac{l-m}{ml}\|\nu_{l}\|_{\partial\mathbb{B}_{1/l}(x_{0})}
 \sup_{z\in S_{m}(\{x_{0}\})}\|G_{K}(z)\|_{\{x_{0}\},\mathcal{J}_{3}(m)}\nu_{m}(z)<\infty,
\end{flalign*}
implying $(z\mapsto \langle\delta_{x_{0}}^{(n)},\widetilde{g}(z,\cdot)\rangle)
\in\mathcal{OV}(\overline{\C}\setminus\{x_{0}\})$.
This means that the path of the integral on the right-hand side of \eqref{satz2.2.9a} can be deformed 
using Cauchy's integral theorem 
in combination with condition $(qV_{\infty})$ (like in \prettyref{prop:pettis-integral} a) and b)) and we get 
with $s:=\min_{j} r_{j}>0$ for $r=(r_{j})$
\begin{align*}
  \frac{1}{2\pi i}\int_{\gamma_{K,k,r}}\langle\delta_{x_{0}}^{(n)},\widetilde{g}(z,\cdot)\rangle\varphi(z)\d z
&=\frac{1}{2\pi i}\int_{\partial\mathbb{B}_{s}(x_{0})}\langle\delta_{x_{0}}^{(n)},\widetilde{g}(z,\cdot)\rangle\varphi(z)\d z\\
&=\frac{1}{2\pi i}\int_{\partial\mathbb{B}_{s}(x_{0})}\widetilde{g}_{z}^{\,(n)}(x_{0})\varphi(z)\d z
\end{align*}
for all $\varphi\in\mathcal{O}\nu_{k}^{-1}(\overline{U_{k}(K)})$. 
Since $g\in\mathcal{O}(\C)$, $g(0)=1$, there is a sequence $(a_{j})_{j\in\N_{0}}$ in $\C$ such that $a_{0}=g(0)=1$ and
\[
g(z-\zeta)=\sum_{j=0}^{\infty}a_{j}(z-\zeta)^{j},\quad z,\zeta\in\C.
\]
Thus the Laurent series of $\widetilde{g}(z,\cdot)=\frac{g(z-\cdot)}{z-\cdot}$ in $\zeta\neq z$ is
\[
\widetilde{g}(z,\zeta)=\frac{1}{z-\zeta}+\sum^{\infty}_{j=1}a_{j}(z-\zeta)^{j-1}
\]
and so
\[
\widetilde{g}_{z}^{\,(n)}(x_{0})=\frac{n!}{(z-x_{0})^{n+1}}+h(z,x_{0})
\]
with an entire function $h(\cdot,x_{0})$. By Cauchy's integral theorem and Cauchy's integral formula for derivatives we have
\begin{align*}
  (H\circ\Theta)(\delta_{x_{0}}^{(n)})(\varphi)
&=\frac{1}{2\pi i}\int_{\partial\mathbb{B}_{s}(x_{0})}{\widetilde{g}_{z}^{\,(n)}(x_{0})\varphi(z)\d z}\\
&=\frac{1}{2\pi i}\int_{\partial\mathbb{B}_{s}(x_{0})}\bigl(\frac{n!}{(z-x_{0})^{n +1}}+h(z,x_{0}) \bigr)\varphi(z)\d z\\
&=\frac{n!}{2\pi i}\int_{\partial\mathbb{B}_{s}(x_{0})}\frac{\varphi(z)}{(z-x_{0})^{n+1}}\d z
 =\varphi^{(n)}(x_{0})=\langle\delta_{x_{0}}^{(n)},\varphi\rangle
\end{align*}
for all $\varphi\in\mathcal{O}\nu_{k}^{-1}(\overline{U_{k}(K)})$, $k\in\N$.
\end{proof}

If $K\cap\{\pm\infty\}$ has isolated points in $K$, e.g.\ $K=\{+\infty\}$, 
then we cannot apply the preceding theorem directly 
since a counterpart for \prettyref{prop:DFS} b) is missing. However, we can make use of the relation \eqref{unabh.H} 
if $\mathcal{OV}^{-1}_{\operatorname{ind}}(\overline{\R})$ is dense in $\mathcal{OV}^{-1}_{\operatorname{ind}}(K)$.

\begin{cor}\label{cor:duality}
Let $K\subset\overline{\R}$ be a non-empty compact set and $\mathcal{V}:=(\nu_{n})_{n\in\N}$ a directed family 
of continuous weights on $\C$ which fulfils $(qV_{\infty})$ and $(qL^{1})$ for $K$ and $\overline{\R}$ as well as 
$(CT.1)$ and $(CT.2)$ for $K$, and $(CT)$ for $\overline{\R}$ with 
$g_{K}=g_{\overline{\R}}$. 
If $E$ is a sequentially complete locally convex Hausdorff space and 
$\mathcal{OV}^{-1}_{\operatorname{ind}}(\overline{\R})$ dense in $\mathcal{OV}^{-1}_{\operatorname{ind}}(K)$,
then the map 
\[
H_{K}\colon \mathcal{OV}(\overline{\C}\setminus K,E)/\mathcal{OV}(\overline{\C},E) 
\to L_{b}(\mathcal{OV}^{-1}_{\operatorname{ind}}(K),E)
\]
is a topological isomorphism with inverse $\Theta_{K}$ and 
\begin{equation}\label{unabh.1}
\Theta_{K}(T)=\Theta_{\overline{\R}}(T),\quad T\in L(\mathcal{OV}^{-1}_{\operatorname{ind}}(K), E).
\end{equation}
\end{cor}
\begin{proof}
$H_{K}$ and $\Theta_{K}$ are well-defined, linear and continuous maps by \prettyref{prop:duality_forward} 
and \prettyref{prop:duality_backward}. 
$H_{\overline{\R}}$ is a topological isomorphism with inverse $\Theta_{\overline{\R}}$ 
by \prettyref{thm:duality} (ii). 
The embedding of $\mathcal{OV}^{-1}_{\operatorname{ind}}(\overline{\R})$ into 
$\mathcal{OV}^{-1}_{\operatorname{ind}}(K)$ 
is continuous and dense, hence defines the embedding of $L(\mathcal{OV}^{-1}_{\operatorname{ind}}(K),E)$ into
$L(\mathcal{OV}^{-1}_{\operatorname{ind}}(\overline{\R}),E)$ (the density of the first embedding implies 
the injectivity of the latter one) and we have
\[
\Theta_{K}(T)=\Theta_{\overline{\R}}(T),\quad T\in L(\mathcal{OV}^{-1}_{\operatorname{ind}}(K),E),
\]
since $g_{\overline{\R}}=g_{K}$. Furthermore, it follows from \eqref{unabh.H} that 
\[
{H_{\overline{\R}}}_{\mid\mathcal{OV}(\overline{\C}\setminus K,E)/\mathcal{OV}(\overline{\C},E)}=H_{K}
\]
on $\mathcal{OV}^{-1}_{\operatorname{ind}}(\overline{\R})$. 
We conclude for every $f\in\mathcal{OV}(\overline{\C}\setminus K,E)/\mathcal{OV}(\overline{\C},E)$ that
\[
 (\Theta_{K}\circ H_{K})(f)=\Theta_{\overline{\R}}(H_{K}(f))=\Theta_{\overline{\R}}(H_{\overline{\R}}(f))=f
\]
and for every $T\in L(\mathcal{OV}^{-1}_{\operatorname{ind}}(K),E)$ that 
\[
 (H_{K}\circ\Theta_{K})(T)=H_{\overline{\R}}(\Theta_{K}(T))=H_{\overline{\R}}(\Theta_{\overline{\R}}(T))=T
\]
by \prettyref{thm:duality}. Thus $H_{K}$ is bijective and $\Theta_{K}$ its inverse.
\end{proof}

\begin{rem}\label{rem:Hausdorff_quotient}
Under the conditions of \prettyref{thm:duality} resp.\ \prettyref{cor:duality} it follows that 
$\mathcal{OV}(\overline{\C}\setminus K,E)/\mathcal{OV}(\overline{\C},E)$ 
is Hausdorff since $E$ and thus $L_{b}(\mathcal{OV}^{-1}_{\operatorname{ind}}(K),E)$ is Hausdorff. 
In particular, $\mathcal{OV}(\overline{\C},E)$ is closed in $\mathcal{OV}(\overline{\C}\setminus K,E)$ 
by \cite[Lemma 22.9, p.\ 254]{meisevogt1997}.
\end{rem}

\begin{cor}\label{cor:duality_example}
Let $E$ be a locally convex Hausdorff space, $K\subset\overline{\R}$ a non-empty compact set, 
$(a_{n})_{n\in\N}$ strictly increasing, 
$a_{n}<0$ for all $n\in\N$ or $a_{n}\geq 0$ for all $n\in\N$
and $\mathcal{V}:=(\exp(a_{n}\mu))_{n\in\N}$ where 
\[
 \mu\colon\C \to [0,\infty),\;\mu(z):=|\re(z)|^{\gamma},
\]
for some $0<\gamma\leq 1$. If 
\begin{enumerate}
 \item [(i)] $K\subset\R$ and $E$ is locally complete, or
 \item [(ii)] $K\cap\{\pm\infty\}$ has no isolated points in $K$ and $E$ is sequentially complete, or 
 \item [(iii)] $K$ is arbitrary, $a_{n}<0$ for all $n\in\N$, $\lim_{n\to\infty}a_{n}=0$, $\gamma=1$ and $E$ sequentially complete,
\end{enumerate}
then the map 
\[
H_{K}\colon \mathcal{OV}(\overline{\C}\setminus K,E)/\mathcal{OV}(\overline{\C},E) 
\to L_{b}(\mathcal{OV}^{-1}_{\operatorname{ind}}(K),E)
\]
is a topological isomorphism with inverse $\Theta_{K}$.
\end{cor}
\begin{proof}
We only need to prove that the conditions of \prettyref{thm:duality} in $(i)$-$(ii)$ 
resp.\ \prettyref{cor:duality} in $(iii)$ 
are fulfilled. For $(iii)$, i.e.\ $K\subset\overline{\R}$ is a non-empty compact set, 
$a_{n}<0$ for all $n\in\N$, $\lim_{n\to\infty}a_{n}=0$ and $\gamma=1$, 
we remark that $\mathcal{OV}^{-1}_{\operatorname{ind}}(\overline{\R})$ is dense in 
$\mathcal{OV}^{-1}_{\operatorname{ind}}(K)$ by \cite[Theorem 2.2.1, p.\ 474]{Kawai} 
and its correction in \cite[Remark, p.\ 247-248]{saburi1985}
(where $\mathcal{OV}^{-1}_{\operatorname{ind}}(\overline{\R})$ is called $\mathscr{P}_{*}$).
Let $M\in\{K,\overline{\R}\}$, $g_{M}\colon\C\to\C$, $g_{M}(z):=\exp(-z^{2})$, and $n\in\N$.

$(qV_{\infty})$: The choices $\mathfrak{I}_{1}(n):=2n$ and 
  \[
   Q:=\overline{U_{\mathfrak{I}_{1}(n)}(M)}\cap\{z\in\C\;|\; |\re(z)|\leq \max(0,\ln(\varepsilon)/(a_{n}-a_{2n}))^{1/\gamma}+n\},
  \]
guarantee that this condition is fulfilled.

$(qL^{1})$: Obviously, $\mu(z)=\mu(|\re(z)|)$ for all $z\in\C$ and with $\mathfrak{I}_{2}(n):=2n$ we have 
\[
 \int_{0}^{\infty}\bigl|\frac{e^{a_{n}\mu(x)}}{e^{a_{\mathfrak{I}_{2}(n)}\mu(x)}}\bigr|\d x
 =\int_{0}^{\infty}e^{(a_{n}-a_{2n})x^{\gamma}}\d x
 =\frac{\Gamma(1/\gamma)}{\gamma |a_{n}-a_{2n}|^{1/\gamma}}
\]
where $\Gamma$ is the gamma function, implying that condition $(qL^{1})$ is satisfied. 
        
$(CT.1)$: Next, we prove that $G_{M}'(z_{0})=-2(z_{0}-\cdot)G_{M}(z_{0})$ for all $z_{0}\in\C$. 
We remark that for all $z=z_{1}+iz_{2}\in\C$
\begin{align*}
   \|G_{M}(z)\|_{M,n}
 &=\sup_{\zeta\in\overline{U_{n}(M)}}e^{-\re((z-\zeta)^{2})}e^{-a_{n}|\re(\zeta)|^{\gamma}}\\
 &\leq \sup_{\zeta_{1}+i\zeta_{2}\in\overline{U_{n}(M)}}e^{-(z_{1}-\zeta_{1})^{2}+(z_{2}-\zeta_{2})^{2}}e^{|a_{n}|(1+|\zeta_{1}|)}\\
 &\leq e^{(|z_{2}|+1/n)^{2}-z_{1}^{2}+|a_{n}|}
       \sup_{\zeta_{1}+i\zeta_{2}\in\overline{U_{n}(M)}}e^{-\zeta_{1}^{2}+(2|z_{1}|+|a_{n}|)|\zeta_{1}|}\\
 &\leq e^{(|z_{2}|+1/n)^{2}-z_{1}^{2}+|a_{n}|}e^{-(|z_{1}|+|a_{n}|/2)^{2}+(2|z_{1}|+|a_{n}|)(|z_{1}|+|a_{n}|/2)}
\end{align*}
and we deduce that $G_{M}(z)\in\mathcal{O}\nu_{n}^{-1}(\overline{U_{n}(M)})$. 
For $\zeta=\zeta_{1}+i\zeta_{2}\in\overline{U_{n}(M)}$, $z_{0}=z_{1}+iz_{2}\in\C$ and $h\in\C$, $0<|h|\leq 1$, we observe that 
\begin{flalign*}
&\hspace{0.37cm} \bigl|\frac{G_{M}(z_{0}+h)(\zeta)-G_{M}(z_{0})(\zeta)}{h}-(-2(z_{0}-\zeta)G_{M}(z_{0})(\zeta))\bigr|
  e^{-a_{n}|\re(\zeta)|^{\gamma}}\\
&=\bigl|\bigl(\frac{e^{-2(z_{0}-\zeta)h-h^{2}}}{h}-\frac{1}{h}+2(z_{0}-\zeta)\bigr)G_{M}(z_{0})(\zeta)\bigr|
  e^{-a_{n}|\re(\zeta)|^{\gamma}}\\
&=\bigl|\bigl(-h+h\sum_{k=2}^{\infty}\frac{1}{k!}(-2(z_{0}-\zeta)-h)^{k}h^{k-2}\bigr)G_{M}(z_{0})(\zeta)\bigr|
  e^{-a_{n}|\re(\zeta)|^{\gamma}}\\
&\leq |h|\bigl(1+\sum_{k=2}^{\infty}\frac{1}{k!}(2|z_{0}-\zeta|+1)^{k}\bigr)|G_{M}(z_{0})(\zeta)|e^{-a_{n}|\re(\zeta)|^{\gamma}}\\
&\leq |h|e^{2|z_{0}-\zeta|+1}e^{-\re((z_{0}-\zeta)^{2})}e^{-a_{n}|\re(\zeta)|^{\gamma}}\\
&\leq |h|e^{2|z_{0}|+2|\zeta_{2}|+2|\zeta_{1}|+1}e^{-(z_{1}-\zeta_{1})^{2}+(z_{2}-\zeta_{2})^{2}}e^{-a_{n}|\zeta_{1}|^{\gamma}}\\
&\leq |h|e^{2|z_{0}|+(2/n)+1-z_{1}^{2}+(|z_{2}|+1/n)^{2}+|a_{n}|}e^{-\zeta_{1}^{2}+(2|z_{1}|+2+|a_{n}|)|\zeta_{1}|}\\
&\leq |h|e^{2|z_{0}|+(2/n)+1-z_{1}^{2}+(|z_{2}|+1/n)^{2}+|a_{n}|}
      e^{-(|z_{1}|+1+|a_{n}|/2)^{2}+(2|z_{1}|+2+|a_{n}|)(|z_{1}|+1+|a_{n}|/2)}=:|h|C_{0},
\end{flalign*}
yielding to 
\[
 \bigl\|\frac{G_{M}(z_{0}+h)-G_{M}(z_{0})}{h}-(-2(z_{0}-\cdot))G_{M}(z_{0})\bigr\|_{M,n}
 \leq |h|C_{0}\underset{h\to 0}{\to} 0.
\]
We conclude that $-2(z_{0}-\cdot)G_{M}(z_{0})\in\mathcal{O}\nu_{n}^{-1}(\overline{U_{n}(M)})$ 
(inequality above and triangle inequality) and $(CT.1)$ holds. 

$(CT.2)$, $(CT.3)$: Let $N\subset M$ be a non-empty compact set. 
We choose $\mathcal{J}_{2}(n):=\mathcal{J}_{3}(n):=2n$ and note that for $\zeta_{1},z_{1}\in\R$
\[
-a_{2n}|\zeta_{1}|^{\gamma}+a_{n}|z_{1}|^{\gamma}\leq |a_{2n}||z_{1}-\zeta_{1}|^{\gamma}\leq |a_{2n}|(1+|z_{1}-\zeta_{1}|).
\]
It follows that
\begin{align*}
  |G_{M}(z)|_{N,n,2n}
&=\sup_{z\in S_{n}(N)}\sup_{\zeta\in\overline{U_{2n}(N)}}e^{-\re((z-\zeta)^{2})}
   e^{-a_{2n}|\re(\zeta)|^{\gamma}}e^{a_{n}|\re(z)|^{\gamma}}\\
&\leq e^{(n+1/(2n))^{2}}\sup_{z_{1}\in\R}\sup_{\zeta_{1}\in\R}e^{-(z_{1}-\zeta_{1})^{2}
                                                                 -a_{2n}|\zeta_{1}|^{\gamma}+a_{n}|z_{1}|^{\gamma}}\\
&\leq e^{(n+1/(2n))^{2}+|a_{2n}|}\sup_{z_{1}\in\R}\sup_{\zeta_{1}\in\R}e^{-(z_{1}-\zeta_{1})^{2}+|a_{2n}||z_{1}-\zeta_{1}|}\\
&\leq e^{(n+1/(2n))^{2}+|a_{2n}|}\sup_{x\in\R}e^{-x^{2}+|a_{2n}|x}\\
&=    e^{(n+1/(2n))^{2}+|a_{2n}|}e^{-(a_{2n}/2)^{2}+a_{2n}^{2}/2},
\end{align*}
which means that $(CT.2)$ and $(CT.3)$ hold.

$(CT.4)$: We set $p:=\mathcal{J}_{4}(n):=2n$ and $m:=\widetilde{\mathcal{J}}_{4}(p):=2p=4n$. 
Then $2\leq p<m$ and for all $z=z_{1}+iz_{2}\in \overline{S_{1/n}}=\{w\in\C\;|\;|\im(w)|\leq (1/n)\}$
\begin{flalign*}
&\hspace{0.37cm}\int_{-\infty}^{\infty}\frac{|g_{M}(z-(t\pm ip))|}{\nu_{m}(t\pm ip)}\nu_{n}(z)\d t\\
&=\int_{-\infty}^{\infty}e^{-\re((z-(t\pm ip))^{2})}e^{a_{n}|z_{1}|^{\gamma}-a_{m}|t|^{\gamma}}\d t
 \leq e^{(z_{2}\mp p)^{2}}\int_{-\infty}^{\infty}e^{-(z_{1}-t)^{2}}e^{|a_{m}||z_{1}-t|^{\gamma}}\d t\\
&\leq e^{((1/n)+p)^{2}+|a_{m}|}\int_{-\infty}^{\infty}e^{-(z_{1}-t)^{2}}e^{|a_{m}||z_{1}-t|}\d t
 =2e^{((1/n)+p)^{2}+|a_{m}|+a_{m}^{2}/4}\int_{-|a_{m}|/2}^{\infty}e^{-t^{2}}\d t\\
&\leq 2\sqrt{\pi}e^{((1/n)+p)^{2}+|a_{m}|+a_{m}^{2}/4},
\end{flalign*}
yielding $(CT.4)$.

$(CT.5)$: For all $z\in\C\setminus M$ it holds that
\begin{align*}
  \sup_{\zeta\in S_{n}(M)}|G_{M}(z)(\zeta)|e^{-a_{n}|\re(\zeta)|^{\gamma}}
&=\sup_{\zeta_{1}+i\zeta_{2}\in S_{n}(M)}e^{-(z_{1}-\zeta_{1})^{2}+(z_{2}-\zeta_{2})^{2}}e^{-a_{n}|\zeta_{1}|^{\gamma}}\\
&\leq e^{(|z_{2}|+n)^{2}-z_{1}^{2}+|a_{n}|}\sup_{\zeta_{1}\in\R}e^{-\zeta_{1}^{2}+(2|z_{1}|+|a_{n}|)\zeta_{1}}\\
&=    e^{(|z_{2}|+n)^{2}-z_{1}^{2}+|a_{n}|}e^{-(|z_{1}|+|a_{n}|/2)^{2}+(2|z_{1}|+|a_{n}|)(|z_{1}|+|a_{n}|/2)}.
\end{align*}
Thus $(CT.5)$ is satisfied.
\end{proof}

The isomorphy $\mathcal{OV}(\overline{\C}\setminus K,E)/\mathcal{OV}(\overline{\C},E) 
\cong L_{b}(\mathcal{OV}^{-1}_{\operatorname{ind}}(K),E)$ 
in \prettyref{cor:duality_example} (iii) is already known for special cases like $E=\C$ 
\cite[Theorem 3.2.1, p.\ 480]{Kawai} and Fr\'echet spaces $E$ \cite[3.9 Satz, p.\ 41]{J} 
but the proof is of homological nature. 
In the special case $K=[a,\infty]$, $a\in\R$, and $E=\C$ the duality in 
\prettyref{cor:duality_example} (iii) was proved in
\cite[Theorem 3.3, p.\ 85-86]{Mori2} and served as an initial point to prove 
\prettyref{cor:duality_example} (iii) for complete $E$ in
\cite[4.1 Theorem, p.\ 41]{ich}.
\section{$(\Omega)$ for $\mathcal{OV}$-spaces on strips with holes}
In this section we derive sufficient conditions such that $\mathcal{OV}(\overline{\C}\setminus K)$ 
satisfies $(\Omega)$ for a non-empty compact set $K\subset\overline{\R}$. 
The basic idea is to prove that, under suitable conditions, the strong dual 
$\mathcal{OV}^{-1}_{\operatorname{ind}}(K)_{b}'$ satisifies $(\Omega)$, 
then we use the duality 
$\mathcal{OV}(\overline{\C}\setminus K)/\mathcal{OV}(\overline{\C})
\cong\mathcal{OV}^{-1}_{\operatorname{ind}}(K)_{b}'$ 
from the preceding section to obtain $(\Omega)$ for $\mathcal{OV}(\overline{\C}\setminus K)$ 
if $\mathcal{OV}(\overline{\C})$ satisfies $(\Omega)$. 
Let us recall that a Fr\'echet space $F$ with an increasing fundamental system of 
seminorms $(\vertiii{\cdot}_{k})_{k\in\N}$ satisfies $(\Omega)$ by \cite[Chap.\ 29, Definition, p.\ 367]{meisevogt1997} if
\[
\forall\; p\in\N\; \exists\; q\in\N\;\forall\; k\in\N\;\exists\; n\in\N,\,C>0\;\forall\; r>0:\;
 U_{q}\subset Cr^n U_k + \frac{1}{r} U_p
\]
where $U_{k}:=\{x\in F \; | \; \vertiii{x}_{k}\leq 1\}$. 

We start with the following helpful observation 
concerning the inductive limit $\mathcal{OV}^{-1}_{\operatorname{ind}}(K)$, 
namely, that the choice of the sequence $(1/n)_{n\in\N}$ for the neighbourhoods $U_{n}(K)=U_{1/(1/n)}(K)$ is irrelevant.

\begin{rem}\label{rem:DFS_sequence}
Let $K\subset \overline{\R}$ be a non-empty compact set, $\mathcal{V}:=(\nu_{n})_{n\in\N}$ a directed family 
of continuous weights on $\C$, $(c_{n})_{n\in\N}$ a strictly decreasing sequence in $\R$ with $c_{n}\leq 1$ for all $n\in\N$ 
and $\lim_{n\to\infty}c_{n}=0$. For $n\in\N$ let
\[
  \mathcal{O}\nu_{n}^{-1}(\overline{U_{1/c_{n}}(K)})
:=\{f\in\mathcal{O}(U_{1/c_{n}}(K))\cap\mathcal{C}(\overline{U_{1/c_{n}}(K)})\;|\; \|f\|_{n,1/c_{n}} < \infty \}
\]
where 
\[
\|f\|_{n,1/c_{n}}:=\sup_{z \in \overline{U_{1/c_{n}}(K)}}|f(z)|\nu_{n}(z)^{-1} 
\]
and the spectral maps for $n,k\in\N$, $n\leq k$, be given by the restrictions
\[
\widetilde{\pi}_{n,k}\colon \mathcal{O}\nu_{n}^{-1}(\overline{U_{1/c_{n}}(K)})\to\mathcal{O}\nu_{k}^{-1}(\overline{U_{1/c_{k}}(K)}), \;
\widetilde{\pi}_{n,k}(f):=f_{\mid U_{1/c_{k}}(K)}.
\]
If $\mathcal{V}$ fulfils $(qV_{\infty})$, then 
\[
 \mathcal{OV}^{-1}_{\operatorname{ind}}(K)
\cong\lim_{\substack{\longrightarrow\\n\in \N}}\mathcal{O}\nu_{n}^{-1}(\overline{U_{1/c_{n}}(K)}).
\]
\end{rem}
\begin{proof}
Follows directly from \prettyref{prop:DFS} a) and \cite[4.2 Satz, p.\ 122]{F/W/Buch}.
\end{proof}

We recall an equivalent description of the property $(\Omega)$.
By \cite[Lemma 29.13, p.\ 369]{meisevogt1997} a Fr\'echet space $F$ with an increasing 
fundamental system of seminorms $(\vertiii{\cdot}_{k})_{k\in\N}$ satisfies $(\Omega)$ if and only if 
\begin{equation}\label{om2}
\forall\; p\in\N\; \exists\; q\in\N\;\forall\; k\in\N\;\exists\;0<\theta<1,\,C>0\;\forall\; y\in F':\;
\|y\|^{\ast}_{q}\leq C {\|y\|^{\ast}_{p}}^{1-\theta} {\|y\|^{\ast}_{k}}^{\theta}
\end{equation}
holds where
\[
\|y\|^{\ast}_{k}:=\sup\{|y(x)|\;|\;\vertiii{x}_{k}\leq 1\}\in\R\cup\{\infty\}
\]
is the dual norm. We introduce the following condition which we need for an application of 
\emph{Hadamard's Three Circles Theorem}.
 
\begin{condH3CT}\label{cond:dual_Omega}
Let $K\subset \overline{\R}$ be a non-empty compact set with $K\cap\{\pm\infty\}\neq\varnothing$ 
and $\mathcal{V}:=(\nu_{n})_{n\in\N}$ a directed family of continuous weights on $\C$.
Let there be a strictly decreasing sequence $(c_{n})_{n\in\N}$ in $\R$ with $c_{n}\leq 1$ for all $n\in\N$ 
and $\lim_{n\to\infty}c_{n}=0$ such that
\begin{align*}
&\forall\; p,q,k\in\N,\,p<q<k\;\exists\;C>0\;\forall\;\zeta\in\R,\,|\zeta|\geq 1+c_{k}^{-1}:\\
&\bigl(\sup_{z\in\C,\,|z-\zeta|\leq c_{k}}\nu_{k}(z)\bigr)^{\theta}\bigl(\sup_{z\in\C,\,|z-\zeta|\leq c_{p}}\nu_{p}(z)\bigr)^{1-\theta}
\leq C\inf_{z\in\C,\,|z-\zeta|\leq c_{q}}\nu_{q}(z)\;\;\text{with}\;\;\theta:=\frac{\ln(c_{p}/c_{q})}{\ln(c_{p}/c_{k})}.
\end{align*}
\end{condH3CT}

We note that $0<\theta<1$ and state the following improvement of \cite[5.21 Lemma, p.\ 88]{ich}.

\begin{lem}\label{lem:dual_Omega}
Let $K\subset\overline{\R}$ be non-empty compact set. If condition $(qV_{\infty})$ and, in addition, $(H3CT)$
if $K\cap\{\pm\infty\}\neq\varnothing$ are fulfilled, then the following holds.
\begin{enumerate}
	\item [a)] 
	\begin{align*}
	 \forall\; p,q,k\in\N,\,p<q<k\;\exists\;&0<\theta<1,\,C>0\;\forall\;f\in\mathcal{O}\nu_{p}^{-1}(\overline{U_{1/c_{p}}(K)}):\\
	 &\|f\|_{q,c_{q}}\leq C\|f\|^{1-\theta}_{p,c_{p}}\|f\|^{\theta}_{k,c_{k}}
	\end{align*}
	with $c_{n}$ from $(H3CT)$ if $K\cap\{\pm\infty\}\neq\varnothing$ 
	resp.\ $c_{n}:=1/n$, $n\in\N$, if $K\subset\R$.
	\item [b)] $\mathcal{OV}^{-1}_{\operatorname{ind}}(K)_{b}'$ satisfies $(\Omega)$.
\end{enumerate}
\end{lem}
\begin{proof}
$a)$ Let $p,q,k\in\N$, $p<q<k$, and $f\in\mathcal{O}\nu_{p}^{-1}(\overline{U_{1/c_{p}}(K)})$. 
Considering the components of $U_{1/c_{p}}(K)$ we have to distinguish three different cases.

$(i)$ Let $Z_{p}$ be a bounded component of $U_{1/c_{p}}(K)$. 
By \prettyref{rem:fin_many_comp+path} a) there are only finitely many components 
$Z_{q}$ of $U_{1/c_{q}}(K)$ with $Z_{q}\subset Z_{p}$. 
For every such component $Z_{q}$ we choose $\zeta\in Z_{q}\cap K$, which exists since $Z_{q}$ is bounded. 
Let $Z_{k}$ be the (unique) component of $U_{1/c_{k}}(K)$ which contains $\zeta$. 
$Z_{p}$ is a proper simply connected subset of $\C$. Thus there exists a biholomorphic map 
$\widetilde{\psi}\colon Z_{p}\to \mathbb{B}_{1}(0)$ with $\widetilde{\psi}(\zeta)=0$ due to the Riemann mapping theorem 
(and M\"obius transformation). In addition, $Z_{p}$ and $\mathbb{B}_{1}(0)$ are Jordan domains 
(for the definition see \cite[2.8.5 Lemma, p.\ 193, 1.8.5 Jordan Curve Theorem, p.\ 68]{B/G/Buch}) 
and so there exists a homeomorphism $\psi\colon\overline{Z_{p}}\to \overline{\mathbb{B}_{1}(0)}$ 
such that $\psi_{\mid Z_{p}}=\widetilde{\psi}$ by \cite[2.8.8 Theorem (Caratheodory), p.\ 195]{B/G/Buch}. 
Since $\psi(\overline{Z_{q}})\subset\psi(Z_{p})=\mathbb{B}_{1}(0)$ and $\psi(\overline{Z_{q}})$ is compact, 
as $\overline{Z_{q}}$ is compact and $\psi$ continuous, there is $0<r_{q}<1$ such that 
$\psi(\overline{Z_{q}})\subset\overline{\mathbb{B}_{r_{q}}(0)}$. 
Moreover, there exists $0<r_{k}<r_{q}$ such that $\overline{\mathbb{B}_{r_{k}}(0)}\subset\psi(Z_{k})$ 
since $0\in\psi(Z_{k})$, $\psi(Z_{k})$ is open by the open mapping theorem (from complex analysis) 
and $\psi(Z_{k})\subset\psi(Z_{q})$. 
The function $u:=f\circ(\psi^{-1})$ is holomorphic on $\mathbb{B}_{1}(0)$ and continuous on $\overline{\mathbb{B}_{1}(0)}$, 
in particular, $|u|$ is subharmonic on $\mathbb{B}_{1}(0)$ and continuous on $\overline{\mathbb{B}_{1}(0)}$. 
Setting 
\[
M(r):=\sup_{|z|=r}|u(z)|,\quad 0<r\leq 1,
\]
we obtain by virtue of \cite[4.4.32 Proposition (Hadamard's Three Circles Theorem), p.\ 338]{B/G/Buch}
\[
\ln(M(r_{q}))\leq \frac{\ln(1/r_{q})}{\ln(1/r_{k})}\ln(M(r_{k}))+ \frac{\ln(r_{q}/r_{k})}{\ln(1/r_{k})}\ln(M(1))
\]
and hence
\[
M(r_{q})\leq M(r_{k})^{\theta}M(1)^{1-\theta}
\]
with $\theta:=\frac{\ln(1/r_{q})}{\ln(1/r_{k})}$. We note that $0<\theta<1$ because $0<r_{k}<r_{q}<1$. 
By the maximum principle we have
\begin{align*}
  M(r_{q})
&=\sup_{|z|\leq r_{q}}|u(z)|
  \geq\inf_{|z|\leq r_{q}}\nu_{q}(\psi^{-1}(z))\sup_{|z|\leq r_{q}}|f(\psi^{-1}(z))|\nu_{q}(\psi^{-1}(z))^{-1}\\
&\underset{\mathclap{\psi(\overline{Z_{q}})\subset \overline{\mathbb{B}_{r_{q}}(0)}}}{\geq}\qquad\;
 \underbrace{\inf_{|z|\leq r_{q}}\nu_{q}(\psi^{-1}(z))}_{=:C_{0}>0}
 \sup_{z\in \overline{Z_{q}}}|f(z)|\nu_{q}(z)^{-1}
\end{align*}
as well as
\begin{flalign*}
&\hspace{0.37cm}M(r_{k})^{\theta}M(1)^{1-\theta}\\
&=\sup_{|z|\leq r_{k}}|u(z)|^{\theta}\sup_{|z|\leq 1}|u(z)|^{1-\theta}\\
&\leq\phantom{\cdot}\bigl(\sup_{|z|\leq r_{k}}\nu_{k}(\psi^{-1}(z))\bigr)^{\theta}
  \bigl(\sup_{|z|\leq r_{k}}|f(\psi^{-1}(z))|\nu_{k}(\psi^{-1}(z))^{-1}\bigr)^{\theta}\\
&\phantom{\leq}
 \cdot\bigl(\sup_{|z|\leq 1}\nu_{p}(\psi^{-1}(z))\bigr)^{1-\theta}
 \bigl(\sup_{|z|\leq 1}|f(\psi^{-1}(z))|\nu_{p}(\psi^{-1}(z))^{-1}\bigr)^{1-\theta}\\
&\underset{\mathclap{\overline{\mathbb{B}_{r_{k}}(0)}\subset\psi(\overline{Z_{k}})}}{\leq}\qquad 
 \underbrace{\bigl(\sup_{|z|\leq r_{k}}\nu_{k}(\psi^{-1}(z))\bigr)^{\theta}
 \bigl(\sup_{|z|\leq 1}\nu_{p}(\psi^{-1}(z))\bigr)^{1-\theta}}_{=:C_{1}}\\
&\phantom{\leq}\cdot
 \bigl(\sup_{z\in\overline{Z_{k}}}|f(z)|\nu_{k}(z)^{-1}\bigr)^{\theta}
 \bigl(\sup_{z\in \overline{Z_{p}}}|f(z)|\nu_{p}(z)^{-1}\bigr)^{1-\theta}
\end{flalign*}
and therefore
\begin{align}\label{thm13.0.0.1}
     \sup_{z\in\overline{Z_{q}}}|f(z)|\nu_{q}(z)^{-1}
&\leq\frac{C_{1}}{C_{0}}\bigl(\sup_{z\in \overline{Z_{k}}}|f(z)|\nu_{k}(z)^{-1}\bigr)^{\theta}
 \bigl(\sup_{z\in \overline{Z_{p}}}|f(z)|\nu_{p}(z)^{-1}\bigr)^{1-\theta}\nonumber\\
&\leq\frac{C_{1}}{C_{0}}\|f\|^{\theta}_{k,c_{k}}\|f\|^{1-\theta}_{p,c_{p}}.
\end{align}

$(ii)$ Let $K\cap\{\pm\infty\}\neq\varnothing$. Let $Z_{p}$ be an unbounded component of $U_{1/c_{p}}(K)$, 
w.l.o.g.\ the real part of $Z_{p}$ is bounded from below and unbounded from above. Let $\zeta\in\R$ such that $\zeta\geq 1+c_{k}^{-1}$. 
Then we have $\overline{\mathbb{B}_{c_{j}}(\zeta)}\subset([c_{j}^{-1},\infty)+i[-c_{j},c_{j}])$ for $j\in\{p,q,k\}$ since 
$c_{p}^{-1}<c_{q}^{-1}<c_{k}^{-1}$ and $c_{j}\leq 1$.
Applying Hadamard's Three Circles Theorem to $u:=|f|$, 
we get $M(c_{q})\leq M(c_{k})^{\theta}M(c_{p})^{1-\theta}$ with $\theta:=\frac{\ln(c_{p}/c_{q})}{\ln(c_{p}/c_{k})}$ 
fulfilling $0<\theta<1$.
Like in $(i)$ we obtain
\[
 M(c_{q})\geq\inf_{|z-\zeta|\leq c_{q}}\nu_{q}(z)\sup_{|z-\zeta|\leq c_{q}}|f(z)|\nu_{q}(z)^{-1}
\]
and
\begin{align*}
     M(c_{k})^{\theta}M(c_{p})^{1-\theta}
&\leq\bigl(\sup_{|z-\zeta|\leq c_{k}}\nu_{k}(z)\bigr)^{\theta}\bigl(\sup_{|z-\zeta|\leq c_{p}}\nu_{p}(z)\bigr)^{1-\theta}\\
&\quad \bigl(\sup_{|z-\zeta|\leq c_{k}}|f(z)|\nu_{k}(z)^{-1}\bigr)^{\theta}\bigl(\sup_{|z-\zeta|\leq c_{p}}|f(z)|\nu_{p}(z)^{-1}\bigr)^{1-\theta}.
\end{align*}
Due to condition $(H3CT)$ there is $C_{2}>0$, independent of $\zeta$, such that
\[
     \sup_{|z-\zeta|\leq c_{q}}|f(z)|\nu_{q}(z)^{-1}
\leq C_{2} \bigl(\sup_{|z-\zeta|\leq c_{k}}|f(z)|\nu_{k}(z)^{-1}\bigr)^{\theta}\bigl(\sup_{|z-\zeta|\leq c_{p}}|f(z)|\nu_{p}(z)^{-1}\bigr)^{1-\theta}
\]
and thus
\begin{flalign}\label{thm13.0.0.5}
&\hspace{0.37cm}\sup_{\substack{z\in\C \\ \d^{|\cdot|}(\{z\},[1+c_{k}^{-1},\infty))\leq c_{q}}}|f(z)|\nu_{q}(z)^{-1}
=\sup_{\substack{\zeta\in\R \\ \zeta\geq 1+c_{k}^{-1}}}\sup_{|z-\zeta|\leq c_{q}}|f(z)|\nu_{q}(z)^{-1}\nonumber\\
&\leq C_{2}\bigl(\sup_{\substack{z\in\C \\ \d^{|\cdot|}(\{z\},[1+c_{k}^{-1},\infty))\leq c_{k}}}|f(z)|\nu_{k}(z)^{-1}\bigr)^{\theta}
      \bigl(\sup_{\substack{z\in\C \\ \d^{|\cdot|}(\{z\},[1+c_{k}^{-1},\infty))\leq c_{p}}}|f(z)|\nu_{p}(z)^{-1}\bigr)^{1-\theta}\nonumber\\
&\leq C_{2}\|f\|^{\theta}_{k,c_{k}}\|f\|^{1-\theta}_{p,c_{p}}.
\end{flalign}

$(iii)$ Let $K\cap\{\pm\infty\}\neq\varnothing$ and $Z_{p}$ be w.l.o.g.\ like in $(ii)$. 
We define $\widetilde{Z}_{p}:=Z_{p}\cap((-\infty,1+c_{k}^{-1})+i\R)$. 
By \prettyref{rem:fin_many_comp+path} a) there are only finitely many components 
$\widetilde{Z}_{q}$ of $U_{1/c_{q}}(K)\cap((-\infty,1+c_{k}^{-1})+i\R)$ with $\widetilde{Z}_{q}\subset\widetilde{Z}_{p}$. 
For every such component $\widetilde{Z}_{q}$ we choose $\zeta\in\widetilde{Z}_{q}\cap(K\cup\{x\in\R\;|\;x>c_{k}^{-1}\})$. 
Let $\widetilde{Z}_{k}$ be the (unique) component of $U_{1/c_{k}}(K)\cap((-\infty,1+c_{k}^{-1})+i\R)$ which contains $\zeta$. 
The rest is analogous to $(i)$ and thus there are $\widetilde{C}_{0}$, $\widetilde{C}_{1}>0$ and $0<\theta<1$ such that
\begin{equation}\label{thm13.0.0.6}
\sup_{z\in\overline{\widetilde{Z}_{q}}}|f(z)|\nu_{q}(z)^{-1}
\leq \frac{\widetilde{C}_{1}}{\widetilde{C}_{0}}\|f\|^{\theta}_{k,c_{k}}\|f\|^{1-\theta}_{p,c_{p}}.
\end{equation}

$(iv)$ First, let us remark the following. Let $B$ be a set, $B_{0}\subset B$, $0<\theta_{0}<\theta_{1}<1$, 
$h\colon B_{0}\to [0,\infty)$, $g\colon B\to [0,\infty)$ and $h\leq g$ on $B_{0}$. Then
\[
    \bigl(\sup_{z\in B_{0}} h(z)\bigr)^{\theta_{1}}\bigl(\sup_{z\in B} g(z)\bigr)^{1-\theta_{1}}
\leq\bigl(\sup_{z\in B_{0}} h(z)\bigr)^{\theta_{0}}\bigl(\sup_{z\in B} g(z)\bigr)^{1-\theta_{0}}.
\]
Now, we take the minimum of all the $\theta$s which appear in $(i)$-$(iii)$. 
There are finitely many of them and denote their minimum again with $\theta$. 
Take the maximum of the constants $\frac{C_{1}}{C_{0}}$, $C_{2}$ and $\frac{\widetilde{C}_{1}}{\widetilde{C}_{0}}$ 
which appear in $(i)$-$(iii)$. 
There are again finitely many of them and denote their maximum with $C$. 
We apply the remark above to $B_{0}:=\overline{U_{1/c_{k}}(K)}$, $B:=\overline{U_{1/c_{p}}(K)}$, 
$h(z):=|f(z)|\nu_{k}(z)^{-1}$ and $g(z):=|f(z)|\nu_{p}(z)^{-1}$. Then we deduce from \eqref{thm13.0.0.1}, 
\eqref{thm13.0.0.5} and \eqref{thm13.0.0.6} that
\[
\|f\|_{q,c_{q}}\leq C\|f\|^{\theta}_{k,c_{k}}\|f\|^{1-\theta}_{p,c_{p}}.
\]

$b)$ We recall \prettyref{rem:DFS_sequence} and identify both inductive limits.
Let $p\in\N$ and choose $q\in\N$, $q>p$. Let $k\in\N$. If $k\leq p$, then we get for any $0<\theta<1$ and all 
$y\in(\mathcal{OV}^{-1}_{\operatorname{ind}}(K)_{b}')'$ by definition of the dual norm 
\[
\|y\|^{\ast}_{q,c_{q}}\underset{p<q}{\leq}\|y\|^{\ast}_{p,c_{p}}={\|y\|^{\ast}_{p,c_{p}}}^{1-\theta}{\|y\|^{\ast}_{p,c_{p}}}^{\theta}
\underset{k<p}{\leq}{\|y\|^{\ast}_{p,c_{p}}}^{1-\theta}{\|y\|^{\ast}_{k,c_{k}}}^{\theta}.
\]
Let $k>p$. If $k\leq q$, we have for any $0<\theta<1$ and 
all $y\in(\mathcal{OV}^{-1}_{\operatorname{ind}}(K)_{b}')'$ by definition of the dual norm 
\[
\|y\|^{\ast}_{q,c_{q}}\underset{k\leq q}{\leq}\|y\|^{\ast}_{k,c_{k}}={\|y\|^{\ast}_{k,c_{k}}}^{1-\theta}{\|y\|^{\ast}_{k,c_{k}}}^{\theta}
\underset{p<k}{\leq}{\|y\|^{\ast}_{p,c_{p}}}^{1-\theta}{\|y\|^{\ast}_{k,c_{k}}}^{\theta}.
\]
Let $k>q$ and $y\in(\mathcal{OV}^{-1}_{\operatorname{ind}}(K)_{b}')'$. 
If $\|y\|^{\ast}_{p,c_{p}}=\infty$, then \eqref{om2} is obviously fulfilled. 
Let $\|y\|^{\ast}_{p,c_{p}}<\infty$. As $\mathcal{OV}^{-1}_{\operatorname{ind}}(K)$ is a DFS-space by \prettyref{prop:DFS} a), 
the sets $B_{n}:=\{f\in\mathcal{O}\nu_{n}^{-1}(\overline{U_{1/c_{n}}(K)})\;|\;\|f\|_{n,c_{n}}\leq 1\}$, $n\in\N$, 
are a fundamental system of bounded sets of $\mathcal{OV}^{-1}_{\operatorname{ind}}(K)$ 
by \cite[Proposition 25.19, p.\ 303]{meisevogt1997} and hence the seminorms
\[
\vertiii{x}_{n}:=\sup_{f\in B_{n}}|x(f)|,\quad x\in\mathcal{OV}^{-1}_{\operatorname{ind}}(K)',
\]
form a fundamental system of seminorms of $\mathcal{OV}^{-1}_{\operatorname{ind}}(K)_{b}'$. 
Furthermore, $\mathcal{OV}^{-1}_{\operatorname{ind}}(K)$ is reflexive and thus there is a unique 
$f\in\mathcal{OV}^{-1}_{\operatorname{ind}}(K)$ such that $y(x)=x(f)$ for all $x\in\mathcal{OV}^{-1}_{\operatorname{ind}}(K)'$.
Then we obtain by \cite[Proposition 22.14, p.\ 256]{meisevogt1997} for all $n\in\N$, $n\geq p$,
\begin{align*}
\infty&>\|y\|^{\ast}_{p,c_{p}}\underset{p\leq n}{\geq}\|y\|^{\ast}_{n,c_{n}}
       =\sup\{|y(x)|\;|\;\vertiii{x}_{n}\leq  1\}
       =\sup\{|x(f)|\;|\;x\in B^{\circ}_{n}\}\\
      &=\inf\{t>0\;|\;f\in tB_{n}\}.
\end{align*}
In particular, this means that $\{t>0\;|\;f\in tB_{n}\}\neq\varnothing$ and 
thus we have $f\in\mathcal{O}\nu_{n}^{-1}(\overline{U_{1/c_{n}}(K)})$ as well as
\[
\|y\|^{\ast}_{n,c_{n}}=\inf\{t>0\;|\;f\in tB_{n}\}=\|f\|_{n,c_{n}}
\]
for all $n\geq p$. So by part $a)$, there are $C>0$ and $0<\theta<1$, only depending on $p$, $q$ and $k$, such that
\[
 \|y\|^{\ast}_{q,c_{q}}=\|f\|_{q,c_{q}}\leq C\|f\|^{1-\theta}_{p,c_{p}}\|f\|^{\theta}_{k,c_{k}}
=C{\|y\|^{\ast}_{p,c_{p}}}^{1-\theta}{\|y\|^{\ast}_{k,c_{k}}}^{\theta}.
\]
\end{proof}

The idea to use Hadamard's Three Circles Theorem in the proof of \prettyref{lem:dual_Omega} a) is taken from 
the proof of \cite[Lemma 5.2 (a)(3), p.\ 263-264]{Vogt1982}.
If $K\subset\R$ is non-empty and compact, \prettyref {lem:dual_Omega} b) is already known. 
Indeed, the space $\mathcal{O}(\C\setminus K)$ satisfies $(\Omega)$ by \cite[Proposition 2.5 (b), p.\ 173]{vogt1983} 
and thus the quotient space 
$\mathcal{O}(\C\setminus K)/\mathcal{O}(\C)$ as well by \cite[Lemma 29.11 (2), p.\ 368]{meisevogt1997}. 
Since $(\Omega)$ is a linear-topological invariant by \cite[Lemma 29.11 (1), p.\ 368]{meisevogt1997},
it follows from $\mathcal{OV}^{-1}_{\operatorname{ind}}(K)_{b}'\cong\mathscr{A}(K)_{b}'
\cong\mathcal{O}(\C\setminus K)/\mathcal{O}(\C)$ by \eqref{eq:duality_non_weighted} that 
$\mathcal{OV}^{-1}_{\operatorname{ind}}(K)_{b}'$ also satisfies $(\Omega)$.

\begin{thm}\label{thm:Omega_strips_with_holes}
Let $K\subset\overline{\R}$ be a non-empty compact set and $\mathcal{OV}(\overline{\C})$ satisfy $(\Omega)$.
Let $(qV_{\infty})$, $(qL^{1})$, $(CT.1)$ and $(CT.2)$ be fulfilled for $K$ and, in addition, 
$(H3CT)$ if $K\cap\{\pm\infty\}\neq\varnothing$. 
If 
\begin{enumerate}
\item [(i)] $K\subset\R$, or $K\cap\{\pm\infty\}$ has no isolated points in $K$ and $(CT.3)$-$(CT.5)$
is fulfilled for $K$, or 
\item [(ii)] $(qV_{\infty})$, $(qL^{1})$ and $(CT)$ are fulfilled for $\overline{\R}$ with $g_{K}=g_{\overline{\R}}$ 
and $\mathcal{OV}^{-1}_{\operatorname{ind}}(\overline{\R})$ is dense in $\mathcal{OV}^{-1}_{\operatorname{ind}}(K)$,
\end{enumerate}
then $\mathcal{OV}(\overline{\C}\setminus K)$ satisfies $(\Omega)$.
\end{thm}
\begin{proof}
The spaces $\mathcal{OV}(\overline{\C}\setminus K)$ and $\mathcal{OV}(\overline{\C})$ 
are Fr\'echet spaces which is easily checked (similar to \cite[3.7 Proposition, p.\ 240]{kruse2018_2}).
By \prettyref{thm:duality} in $(i)$ resp.\ \prettyref{cor:duality} in $(ii)$ 
$\mathcal{OV}(\overline{\C}\setminus K)/\mathcal{OV}(\overline{\C})$ 
is topologically isomorphic to $\mathcal{OV}^{-1}_{\operatorname{ind}}(K)_{b}'$, 
in particular, the quotient is a Fr\'echet space as $\mathcal{OV}^{-1}_{\operatorname{ind}}(K)$ is 
a DFS-space by \prettyref{prop:DFS} a).
Since $(\Omega)$ is a linear-topological invariant by \cite[Lemma 29.11 (1), p.\ 368]{meisevogt1997}, 
$\mathcal{OV}(\overline{\C}\setminus K)/\mathcal{OV}(\overline{\C})$ satisfies $(\Omega)$ 
due to \prettyref{lem:dual_Omega} b). The sequence
\[
0\to\mathcal{OV}(\overline{\C})\overset{i}{\to}\mathcal{OV}(\overline{\C}\setminus K)
\overset{q}{\to}\mathcal{OV}(\overline{\C}\setminus K)/\mathcal{OV}(\overline{\C})\to 0
\]
is an exact sequence of Fr\'echet spaces where $i$ means the inclusion and $q$ the quotient map. 
$\mathcal{OV}(\overline{\C})$ satisfies $(\Omega)$ by assumption
and $\mathcal{OV}(\overline{\C}\setminus K)/\mathcal{OV}(\overline{\C})$ as well, 
thus $\mathcal{OV}(\overline{\C}\setminus K)$ by \cite[1.7 Lemma, p.\ 230]{VW}, too.
\end{proof}

Conditions for $\mathcal{OV}(\overline{\C})$ to satisfy $(\Omega)$ can be found in 
\cite[Theorem 10, p.\ 14]{kruse2019_1}, \cite[Corollary 13, p.\ 17]{kruse2019_1} 
and more general in \cite[3.1 Proposition]{debrouwere2021}.
In particular, the preceding theorem generalises \cite[5.22 Theorem, p.\ 92]{ich} 
which is case $(ii)$ of the following corollary.

\begin{cor}\label{cor:Omega_strips_with_holes}
Let $K\subset\overline{\R}$ be a non-empty compact set, 
$(a_{n})_{n\in\N}$ strictly increasing, $a_{n}< 0$ for all $n\in\N$ or $a_{n}\geq 0$ for all $n\in\N$, 
$\lim_{n\to\infty}a_{n}=0$ or $\lim_{n\to\infty}a_{n}=\infty$
and $\mathcal{V}:=(\exp(a_{n}\mu))_{n\in\N}$ where 
\[
 \mu\colon\C \to [0,\infty),\;\mu(z):=|\re(z)|^{\gamma},
\]
for some $0<\gamma\leq 1$. If 
\begin{enumerate}
 \item [(i)] $K\subset\R$, or $K\cap\{\pm\infty\}$ has no isolated points in $K$, or
 \item [(ii)] $K$ is arbitrary, $a_{n}< 0$ for all $n\in\N$, $\lim_{n\to\infty}a_{n}=0$ and $\gamma=1$,
\end{enumerate}
then $\mathcal{OV}(\overline{\C}\setminus K)$ satisfies $(\Omega)$.
\end{cor}
\begin{proof}
By \cite[Corollary 14, p.\ 18]{kruse2019_1} $\mathcal{OV}(\overline{\C})$ satisfies $(\Omega)$. Due to \prettyref{thm:Omega_strips_with_holes} and (the proof of) \prettyref{cor:duality_example} 
we only need to check that $(H3CT)$ is fulfilled if $K\cap\{\pm\infty\}\neq\varnothing$. 
Let $c_{n}:=\exp(1/a_{n})$ for all $n\in\N$ if $a_{n}<0$ for all $n\in\N$ 
and $c_{n}:=\exp(-a_{n})$ for all $n\in\N$ if $a_{n}\geq 0$ for all $n\in\N$. Then $(c_{n})$ is a strictly decreasing sequence, 
$c_{n}\leq 1$ for all $n\in\N$ and $\lim_{n\to\infty}c_{n}=0$. Let $p,q,k\in\N$ such that $p<q<k$ and $\theta:=\frac{\ln(c_{p}/c_{q})}{\ln(c_{p}/c_{k})}$.
Let $\zeta\in\R$ with $|\zeta|\geq 1+c_{k}^{-1}$.
For $z\in\C$ with $|z-\zeta|\leq c_{n}$, $n\in\{p,q,k\}$, we deduce from the inequalities
\[
 ||\zeta|-|\re(z)-\zeta||^{\gamma}\leq |\re(z)|^{\gamma}\leq (|\re(z)-\zeta|+|\zeta|)^{\gamma}\leq (c_{n}+|\zeta|)^{\gamma}
\]
and 
\[
 |\zeta|-|\re(z)-\zeta|\geq |\zeta|-c_{n}\geq 1+c_{k}^{-1}-c_{n}\geq c_{k}^{-1}>0
\]
that 
\[
 \inf_{z\in\C,\,|z-\zeta|\leq c_{q}}e^{a_{q}|\re(z)|^{\gamma}}\geq e^{a_{q}(c_{q}+|\zeta|)^{\gamma}}
\]
and
\[
 \sup_{z\in\C,\,|z-\zeta|\leq c_{k}}e^{\theta a_{k}|\re(z)|^{\gamma}}\sup_{z\in\C,\,|z-\zeta|\leq c_{p}}e^{(1-\theta) a_{p}|\re(z)|^{\gamma}}
 \leq e^{\theta a_{k}(|\zeta|-c_{k})^{\gamma}+(1-\theta) a_{p}(|\zeta|-c_{p})^{\gamma}},
\]
if $a_{n}<0$, as well as 
\[
 \inf_{z\in\C,\,|z-\zeta|\leq c_{q}}e^{a_{q}|\re(z)|^{\gamma}}\geq e^{a_{q}(|\zeta|-c_{q})^{\gamma}}
\]
and
\[
 \sup_{z\in\C,\,|z-\zeta|\leq c_{k}}e^{\theta a_{k}|\re(z)|^{\gamma}}\sup_{z\in\C,\,|z-\zeta|\leq c_{p}}e^{(1-\theta) a_{p}|\re(z)|^{\gamma}}
 \leq e^{\theta a_{k}(c_{k}+|\zeta|)^{\gamma}+(1-\theta) a_{p}(c_{p}+|\zeta|)^{\gamma}},
\]
if $a_{n}\geq 0$. Now, we only need to prove that there is $C>0$ such that 
\[
 e^{\theta a_{k}(|\zeta|-c_{k})^{\gamma}+(1-\theta) a_{p}(|\zeta|-c_{p})^{\gamma}}\leq C e^{a_{q}(c_{q}+|\zeta|)^{\gamma}},\quad a_{n}<0,
\]
resp.\ 
\[
 e^{\theta a_{k}(c_{k}+|\zeta|)^{\gamma}+(1-\theta) a_{p}(c_{p}+|\zeta|)^{\gamma}}\leq C e^{a_{q}(|\zeta|-c_{q})^{\gamma}},\quad a_{n}\geq 0.
\]
If $a_{n}<0$, we observe that 
\begin{flalign*}
 &\hspace{0.37cm}\theta a_{k}(|\zeta|-c_{k})^{\gamma}+(1-\theta) a_{p}(|\zeta|-c_{p})^{\gamma}-a_{q}(c_{q}+|\zeta|)^{\gamma}\\
 &\leq \theta a_{k}(|\zeta|-c_{p})^{\gamma}+(1-\theta) a_{p}(|\zeta|-c_{p})^{\gamma}-a_{q}(|\zeta|-c_{p})^{\gamma}-a_{q}(c_{p}+c_{q})^{\gamma}\\
 &= (\theta a_{k}+(1-\theta) a_{p}-a_{q})(|\zeta|-c_{p})^{\gamma}-a_{q}(c_{p}+c_{q})^{\gamma}
\end{flalign*}
and, if $a_{n}\geq 0$, that 
\begin{flalign*}
 &\hspace{0.37cm}\theta a_{k}(c_{k}+|\zeta|)^{\gamma}+(1-\theta) a_{p}(c_{p}+|\zeta|)^{\gamma}-a_{q}(|\zeta|-c_{q})^{\gamma}\\
 &\leq \theta a_{k}(c_{p}+|\zeta|)^{\gamma}+(1-\theta) a_{p}(c_{p}+|\zeta|)^{\gamma}-a_{q}|||\zeta|+c_{p}|^{\gamma}-|c_{p}+c_{q}|^{\gamma}|\\
 &\leq (\theta a_{k}+(1-\theta) a_{p}-a_{q})(c_{p}+|\zeta|)^{\gamma}+a_{q}(c_{p}+c_{q})^{\gamma}.
\end{flalign*} 
What remains to be shown is that 
\begin{equation}\label{eq:cor_Omega_strips_with_holes}
 0\geq \theta a_{k}+(1-\theta) a_{p}-a_{q}
\end{equation}
because then we are done with $C:=\exp(|a_{q}|(c_{p}+c_{q})^{\gamma})$. If $a_{n}<0$, then 
\[ 
\theta=\frac{\ln(c_{p}/c_{q})}{\ln(c_{p}/c_{k})}=\frac{(1/a_{p})-(1/a_{q})}{(1/a_{p})-(1/a_{k})}
=\frac{a_{k}(a_{q}-a_{p})}{a_{q}(a_{k}-a_{p})}
\]
and \eqref{eq:cor_Omega_strips_with_holes} is equivalent to 
\[
 0\geq \frac{a_{k}^{2}(a_{q}-a_{p})}{a_{q}(a_{k}-a_{p})}+\bigl(1-\frac{a_{k}(a_{q}-a_{p})}{a_{q}(a_{k}-a_{p})}\bigr)a_{p}-a_{q},
\]
which holds if and only if 
\begin{align*}
 0&\leq a_{k}^{2}(a_{q}-a_{p})+\bigl(a_{q}(a_{k}-a_{p})-a_{k}(a_{q}-a_{p})\bigr)a_{p}-a_{q}^{2}(a_{k}-a_{p})\\
  &=a_{k}^{2}(a_{q}-a_{p})+(a_{k}-a_{q})a_{p}^{2}-a_{q}^{2}(a_{k}-a_{q}+a_{q}-a_{p})\\
  &=(a_{k}^{2}-a_{q}^{2})(a_{q}-a_{p})+(a_{k}-a_{q})(a_{p}^{2}-a_{q}^{2})\\
  &=(a_{k}-a_{q})(a_{k}+a_{q})(a_{q}-a_{p})-(a_{k}-a_{q})(a_{q}-a_{p})(a_{p}+a_{q})
\end{align*}
as $a_{q}(a_{k}-a_{p})<0$. Since $a_{k}-a_{q}>0$ and $a_{q}-a_{p}>0$, this is equivalent to 
\[
 0\leq (a_{k}+a_{q})-(a_{p}+a_{q})=a_{k}-a_{p},
\]
which is true. If $a_{n}\geq 0$, then 
\[
 \theta=\frac{\ln(c_{p}/c_{q})}{\ln(c_{p}/c_{k})}=\frac{a_{q}-a_{p}}{a_{k}-a_{p}}
\]
and \eqref{eq:cor_Omega_strips_with_holes} is equivalent to 
\[
 0\geq \frac{a_{q}-a_{p}}{a_{k}-a_{p}}a_{k}+\bigl(1-\frac{a_{q}-a_{p}}{a_{k}-a_{p}}\bigr)a_{p}-a_{q},
\]
which holds, as $a_{k}-a_{p}>0$, if and only if 
\begin{align*}
 0&\geq (a_{q}-a_{p})a_{k}+\bigl(a_{k}-a_{p}-(a_{q}-a_{p})\bigr)a_{p}-(a_{k}-a_{p})a_{q}\\
  &= a_{q}a_{k}-a_{p}a_{k}+a_{k}a_{p}-a_{q}a_{p}-a_{k}a_{q}+a_{p}a_{q}
   =0.
\end{align*}
\end{proof}
\section{Surjectivity of the Cauchy-Riemann operator}
In our last section we prove our main result on the surjectivity of the Cauchy-Riemann operator 
on $\mathcal{EV}(\overline{\C}\setminus K,E)$ for non-empty compact $K\subset\overline{\R}$. 
This is done by using the results obtained so far and splitting theory. 
We recall that a Fr\'echet space $(F,(\vertiii{\cdot}_{k})_{k\in\N})$ satisfies $(DN)$
by \cite[Chap.\ 29, Definition, p.\ 359]{meisevogt1997} if
\[
\exists\;p\in\N\;\forall\;k\in\N\;\exists\;n\in\N,\,C>0\;\forall\;x\in F:\;
\vertiii{x}^{2}_{k}\leq C\vertiii{x}_{p}\vertiii{x}_{n}.
\]
A \emph{PLS-space} is a projective limit $X=\lim\limits_{\substack{\longleftarrow\\N\in\N}}X_{N}$, where the
inductive limits $X_{N}=\lim\limits_{\substack{\longrightarrow\\n\in \N}}(X_{N,n},\vertiii{\cdot}_{N,n})$ are DFS-spaces, and it satisfies $(PA)$ if
\begin{gather*}
\forall \;N\;\exists\;  M\;  \forall\;  K\;  \exists\;  n\;  \forall\;  m\;  \forall\;  \eta >0\;  \exists\;  k,C,r_0 >0\;  \forall\;  r>r_0\; \forall\; x'\in X'_{N}:\\
	\vertiii{x'\circ i^{M}_{N}}^{\ast}_{M,m}\leq C\bigl(r^{\eta}\vertiii{x'\circ i^{K}_{N}}^{\ast}_{K,k}+\frac{1}{r}\vertiii{x'}^{\ast}_{N,n}\bigr)
\end{gather*}
where $\vertiii{\cdot}^{\ast}$ denotes the dual norm of $\vertiii{\cdot}$ and $i^{M}_{N}$, $i^{K}_{N}$ the linking maps 
(see \cite[Section 4, Eq.\ (24), p.\ 577]{Dom1}).

\begin{thm}[{\cite[Theorem 5, p.\ 7-8]{kruse2019_1}}]\label{thm:surj_CR_DN_PA}
Let $\mathcal{EV}(\Omega)$ be a Schwartz space and $\mathcal{EV}_{\overline{\partial}}(\Omega)$ 
a nuclear subspace satisfying property $(\Omega)$. 
Assume that the scalar-valued operator $\overline{\partial}\colon\mathcal{EV}(\Omega)\to\mathcal{EV}(\Omega)$ 
is surjective. Moreover, if
\begin{enumerate}
\item [a)] $E:=F_{b}'$ where $F$ is a Fr\'echet space over $\C$ satisfying $(DN)$, or 
\item [b)] $E$ is an ultrabornological PLS-space over $\C$ satisfying $(PA)$, 
\end{enumerate}
then
\[
\overline{\partial}^{E}\colon \mathcal{EV}(\Omega,E)\to\mathcal{EV}(\Omega,E)
\]
is surjective.
\end{thm}

The case $a)$ is included in case $b)$ if $F$ is a Fr\'echet-Schwartz space by \cite[Remark 2, p.\ 6]{kruse2019_1}. 
If $E$ is a Fr\'echet space over $\C$ the preceding theorem is also valid but the assumption 
that $\mathcal{EV}_{\overline{\partial}}(\Omega)$ satifies property $(\Omega)$ is not needed 
(see \cite[4.9 Corollary, p.\ 21]{kruse2018_5}). 

\begin{cor}\label{cor:surj_CR_DN_PA_exa}
Let $K\subset\overline{\R}$ be a non-empty compact set, 
$(a_{n})_{n\in\N}$ strictly increasing, $a_{n}<0$ for all $n\in\N$, $\lim_{n\to\infty}a_{n}=0$
and $\mathcal{V}:=(\exp(a_{n}\mu))_{n\in\N}$ where 
\[
 \mu\colon\C \to [0,\infty),\;\mu(z):=|\re(z)|^{\gamma},
\]
for some $0<\gamma\leq 1$. If 
\begin{enumerate}
 \item [(i)] $K\subset\R$, or $K\cap\{\pm\infty\}$ has no isolated points in $K$, or
 \item [(ii)] $K$ is arbitrary and $\gamma=1$,
\end{enumerate}
and 
\begin{enumerate}
\item [a)] $E:=F_{b}'$ where $F$ is a Fr\'echet space over $\C$ satisfying $(DN)$, or 
\item [b)] $E$ is an ultrabornological PLS-space over $\C$ satisfying $(PA)$, 
\end{enumerate}
then
\[
\overline{\partial}^{E}\colon \mathcal{EV}(\overline{\C}\setminus K,E)\to\mathcal{EV}(\overline{\C}\setminus K,E)
\]
is surjective.
\end{cor}
\begin{proof}
We only need to check that the conditions of \prettyref{thm:surj_CR_DN_PA} are fulfilled. 
$\mathcal{EV}(\overline{\C}\setminus K)$ is nuclear, in particular a Schwartz space, and thus its subspace 
$\mathcal{EV}_{\overline{\partial}}(\overline{\C}\setminus K)$ as well by \cite[Theorem 3.1, p.\ 188]{kruse2018_4}, 
\cite[2.8 Example (ii), p.\ 179]{kruse2018_4}, \cite[Remark 2.7, p.\ 178-179]{kruse2018_4} 
and \cite[Remark 2.3 (b), p.\ 177]{kruse2018_4}. 
Furthermore, $\mathcal{EV}_{\overline{\partial}}(\overline{\C}\setminus K)=\mathcal{OV}(\overline{\C}\setminus K)$ 
by \cite[Proposition 7 (b), p.\ 11]{kruse2019_1} and \cite[Example 6, p.\ 11]{kruse2019_1}. 
Due to \prettyref{cor:Omega_strips_with_holes} the space $\mathcal{OV}(\overline{\C}\setminus K)$ 
satisfies $(\Omega)$. The Cauchy-Riemann operator $\overline{\partial}\colon\mathcal{EV}(\overline{\C}\setminus K)\to\mathcal{EV}(\overline{\C}\setminus K)$ in the $\C$-valued case is surjective by \cite[Corollary 5.6, p.\ 27]{kruse2018_5} 
which follows from \cite[Example 5.7 (a), p.\ 27-28]{kruse2018_5} 
in the case that $K\subset\R$ or $K\cap\{\pm\infty\}$ has no isolated points in $K$.
If $K\cap\{\pm\infty\}$ has isolated points in $K$, then the proof that the conditions of 
\cite[Corollary 5.6, p.\ 27]{kruse2018_5}
are fulfilled is verbatim as in \cite[Example 5.7 (a), p.\ 27-28]{kruse2018_5}. Hence 
all conditions of \prettyref{thm:surj_CR_DN_PA} are fulfilled. 
\end{proof}

\prettyref{cor:surj_CR_DN_PA_exa}, together with \cite[Corollary 18, p.\ 21]{kruse2019_1} ($K=\varnothing$), 
generalises \cite[5.24 Theorem, p.\ 95]{ich} which is case $(ii)$. 

\subsection*{Acknowledgements}
The present paper is a generalisation of parts of Chapter 4 and 5 of the author's PhD thesis \cite{ich}, 
written under the supervision of M.\ Langenbruch. The author is deeply grateful to him for his support and advice. 
Further, it is worth to mention that some of the results appearing in the PhD thesis 
and thus their generalised counterparts in this work are essentially due to him.
\bibliography{biblio}
\bibliographystyle{plainnat}
\end{document}